\documentclass{amsart}

\usepackage{amssymb}

\newtheorem{theorem}{Theorem}[section]
\newtheorem{lemma}[theorem]{Lemma}

\newtheorem{claim}[theorem]{Claim}
\newtheorem{remark}[theorem]{Remark}

\numberwithin{equation}{section}
\newcommand{\Ai}{\text{Ai\,}}

\newcommand{\Tr}{\text{tr\,}}
\newcommand{\re}{\text{Re\,}}
\newcommand{\im}{\text{Im\,}}

\begin{document}
\setcounter{page}{1}

\thanks{Supported by the G\"oran Gustafsson Foundation (KVA) and the Swedish  
Research Council (VR)}


\title[Universality under weak moment conditions]
{Universality for certain Hermitian Wigner matrices under weak moment conditions}
\author[K.~Johansson]{Kurt Johansson}

\address{
Department of Mathematics,
Royal Institute of Technology,
SE-100 44 Stockholm, Sweden}

\email{kurtj@kth.se}

\begin{abstract}
We study the universality of the local eigenvalue statistics of Gaussian divisible Hermitian
Wigner matrices. These random matrices are obtained by adding an independent GUE matrix to an
Hermitian random matrix with independent elements, a Wigner matrix.
We prove that Tracy-Widom universality holds at the edge in this class of random matrices under 
the optimal moment condition that there is a uniform bound on
the fourth moment of the matrix elements. Furthermore, we show that universality holds in the bulk 
for Gaussian divisible Wigner matrices if we just assume finite second moments. 
\end{abstract}

\maketitle

\section{Introduction and results}\label{sect1}
\subsection{Introduction}\label{sect1.1}
An Hermitian Wigner matrix is a random Hermitian matrix  with independent elements respecting
the Hermitian symmetry. The local eigenvalue statistics of these random matrices is expected to
be universal in the sense that it is independent of the distribution of the individual matrix
elements, at least under suitable assumptions on the moments of the elements. There are two basic cases.
We can either look in the bulk of the spectrum or at the edge around the largest eigenvalue.
It is conjectured that, if we assume that the real and imaginary parts of the elements all have mean
value zero, variance $\sigma^2>0$ and that there is a uniform bound on the fourth moment, then the 
appropriately scaled eigenvalue point process at the edge should converge to the Airy kernel 
point process. Furthermore
the largest eigenvalue should asymptotically fluctuate according to the Tracy-Widom distribution.
This problem is still open, but there are results under stronger moment assumptions. The 
breakthrough result by Soshnikov, \cite{So1}, showed that the result is true if the 
distribution is symmetric and has sub-gaussian tails. Soshnikov's result is based on moment methods. 
The condition on the moments has been weakened to $18+\epsilon$ moments (or $36+\epsilon$ moments, see
\cite{ABP}) in \cite{Ru}.

In the bulk it is expected that the local 
eigenvalue point process converges to the sine-kernel point process.
The exact conditions needed for this to be true are not clear. The result in the bulk was proved for
a sub-class of Wigner matrices, so called Gaussian divisible Hermitian Wigner matrices in 
\cite{JoUniv}. A Gaussian divisible Hermitian Wigner matrix is an Hermitian Wigner matrix
$W$ of the form $W=X+\sqrt{\kappa} V$, where $X$ is an Hermitian Wigner matrix and $V$ an independent
GUE matrix. In \cite{JoUniv} it was assumed that the elements of $X$ have uniformly bounded $6+\epsilon$
moments. Spectacular progress has recently been made on this problem by Tao and Vu, \cite{TV1},
with their four-moment theorem, and by Erd\"os, Ramirez, Schlein and H.-T Yau using a different approach,
\cite{ERSY}.
Tao and Vu assume subexponential tails for the distribution of the matrix elements.
Erd\"os, Ramirez, Schlein and H.-T Yau make rather strong regularity assumptions on the distribution
and parts of the argument use methods related to the approach in \cite{JoUniv} and this paper.
A combined effort, \cite{ERSTVY}, removed some of the assumptions in \cite{TV1}. Thus,
the universality result in the bulk is now established under the assumption of subexponential decay 
of the tails of the distributions of the matrix elements \footnote{Very recently \cite{TV3} the
assumption on the distribution has been reduced to a finite but large number of moments} .

Very recently, Tao and Vu, \cite{TV2}, also generalized Soshnikov's result using an approach 
analogous to that in their paper on bulk universality. They obtain universality at the edge under the
assumption of subexponential deacy and vanishing third moments. The result in this paper can be used to
remove this third moment assumption, see theorem \ref{thm1.3}.

The four-moment theorem indicates that the class of Gaussian divisible Wigner matrices is a good
testing ground for what we can expect for General Wigner matrices.
In this paper we therefore return to the case of Gaussian divisible Hermitian Wigner matrices with the aim of 
establishing universality results within this class under weak moment conditions. In particular, we
prove universality at the edge under the optimal assumption that the fourth moment is finite.
It is known that if we have fewer than four moments then the behaviour around the largest eigenvalue is 
instead described by a Poisson process, see \cite{ABP}, \cite{So2}, \cite{BBP}.

We also show universality in the bulk within the class of Gaussian divisible Hermitian Wigner
matrices under the assumption that the second moment is finite. It is not claer that this is the optimal 
condition. Rather, close to the origin we should still expect sine-kernel universality even if the
second moment is infinite, see \cite{CB}.

The results are obtained using a development of the techniques in \cite{JoUniv} which were based on a 
contour integral formula for a correlation kernel from \cite{BrHi}. In \cite{JoUniv} an important
tool was a concentration of measure estimate from \cite{GuZe}, which led to a uniform estimate  
of the Stieltjes transform of the empirical spectral measure of $X$ in a region of the complex plane.
Here, due to the weak moment assumptions we are unable to use this result and we have to be satisfied with
weaker pointwise control of the Stieltjes transform. This requires a modification of the 
analysis in \cite{JoUniv} and a more careful choice of contours, since we have do not have
the same good contol of the empirical spectral measure of $X$. The pointwise control of expectations
of the Stieltjes transform that we need is adapted from \cite{Bai1} and \cite{BMT}.

\subsection{Results}\label{sect1.2}
We turn now to precise statements of our results. The $n\times n$ random matrix $X$ is an {\it Hermitian
Wigner matrix} if $X=(x_{ij})$ is Hermitian, $\re x_{ij}$, $\im x_{ij}$, $1\le i<j\le n$ and $X_{jj}$,
$1\le j\le n$ are all independent and satisfy
\begin{itemize}
\item[(i)] $\mathbb{E}[\re x_{ij}]=\mathbb{E}[\im x_{ij}]=0,\quad 1\le i\le j\le n$,

\item[(ii)] $\mathbb{E}[(\re x_{ij})^2]=\mathbb{E}[(\im x_{ij})^2]=\sigma^2/2,\quad 1\le i< j\le n$,

\item[(iii)] $\mathbb{E}[x_{jj}^2]=\sigma^2$.

\end{itemize}

We will say that $W$ is a Gaussian divisible Hermitian Wigner matrix if it can be written
\begin{equation}\label{1.1}
W=X+\sqrt{\kappa} V,
\end{equation}
where $X$ is an Hermitian Wigner matrix, $\kappa$ a positive constant and $V$ an independent
GUE-matrix. We take the GUE-measure to be
\begin{equation}
\frac 1{Z_n}e^{-\Tr V^2/2}\,dV.
\notag
\end{equation}
Without loss of generality we can choose the variance $\sigma^2=1/4$.

Let $\{\lambda_j\}$ be the eigenvalues of $\sqrt{n}W$. Then the sequence 
$\{\lambda_j/\sqrt{n}\}$ is asymptotically
distributed according to {\it the Wigner semi-circle law},
\begin{equation}\label{1.2}
\rho(x)=\frac 2{\pi(1+4\kappa)}\sqrt{(1+4\kappa-x^2)_+}.
\end{equation}
Let $C_c(\mathbb{R})$ denote the set of all continuous functions with compact support, and 
$C_c^+(\mathbb{R})$ the subset of $C_c(\mathbb{R})$ of non-negative functions. For $b>0$ let
\begin{equation}\label{1.3}
K^b_{\text{sine}}(u,v)=\frac{\sin b(u-v)}{\pi(u-v)}
\end{equation}
be the {\it sine kernel} with density $b/\pi$. The {\it sine-kernel point process} on infinite
point configurations $\{\mu_j\}$ on the real line is the determinantal point process
defined by
\begin{equation}\label{1.4}
\mathbb{E}^b_{\text{sine}}\left[\exp(-\sum_j\psi(\mu_j))\right]=\det(I-\phi^{1/2}
K^b_{\text{sine}}\phi^{1/2})
\end{equation}
for all $\psi\in C^+_c(\mathbb{R})$, where $\phi=1-e^{-\psi}$. Here, the right hand side is the
Fredholm determinant on $L^2(\mathbb{R})$ with kernel $\phi^{1/2}K^b_{\text{sine}}\phi^{1/2}$.

\begin{theorem}\label{thm1.1}
Let $W$ be a Gaussian divisible Hermitian Wigner matrix with finite second moments as in (\ref{1.1}),
and let $\{\lambda_j\}$ be the eigenvalues of $\sqrt{n}W$. Assume that $d_n/n\to d$ as $n\to\infty$, 
where $|d|<\sqrt{1+4\kappa}$, and let
\begin{equation}\label{1.5}
\beta=\frac 2{1+4\kappa}\sqrt{1+4\kappa-d^2}.
\end{equation}
Then,
\begin{equation}\label{1.6}
\lim_{n\to\infty}\mathbb{E}\left[\exp(-\sum_{j=1}^n\psi(\lambda_j-d_n))\right]=
\mathbb{E}^{\beta}_{\text{{\rm sine}}}\left[\exp(-\sum_j\psi(\mu_j))\right]
\end{equation}
for all $\psi\in C^+_c(\mathbb{R})$.
\end{theorem}

The theorem will be proved in section \ref{sect2.2}. The theorem shows that the appropriately 
scaled eigenvalue point process converges weakly in the bulk, i.e. in the interior of
the support of the semi-circle law, (\ref{1.2}), to the sine kernel point process with
density given by the semi-circle law. This theorem is an extension of the main result theorem in 
\cite{JoUniv}, see also \cite{BePe}.

We turn now to the edge behaviour. It is known that if the matrix elements are heavy-tailed with
no fourth moment, then the eigenvalue point process at the edge converges to a Poisson point 
process with a certain density, see \cite{ABP}, \cite{So2} and \cite{BBP}. Thus, in order to
get the same edge behaviour as for GUE we have to assume at least that the fourth moment
is finite. It is known, see \cite{Bai2}, that finite fourth moments is necessary and sufficient for
the largest eigenvalue to converge to the edge of the support of the semi-circle. We will show that
within the class of Gaussian divisible Wigner matrices finite fourth moments suffices 
for Tracy-Widom asymptotics.

The eigenvalue statistics of a GUE-matrix at the edge is described by the Airy kernel
point process. The {\it Airy kernel} is defined by
\begin{equation}\label{1.7}
A(x,y)=\int_0^\infty\Ai(x+t)\Ai(y+t)\,dt=
\frac{\Ai(x)\Ai'(y)-\Ai'(x)\Ai(y)}{x-y}.
\end{equation}
The {\it Airy kernel point process} on infinite point configurations $\{\mu_j\}$ on the real line is 
the determinantal point process defined by
\begin{equation}\label{1.8}
\mathbb{E}_{\text{Airy}}\left[\exp(-\sum_j\psi(\mu_j))\right]=\det(I-\phi^{1/2}A\phi^{1/2}),
\end{equation}
for all $\psi\in C^+_c(\mathbb{R})$, where $\phi=1-e^{-\psi}$. The Airy kernel point process
has almost surely a last particle $\mu_{\max}$ whose distribution is given by the
{\it Tracy-Widom distribution},
\begin{equation}\label{1.9}
\mathbb{P}_{\text{Airy}}[\mu_{\max}\le t]=F_{\text{TW}}(t)=\det(I-A)_{L^2(t,\infty)}.
\end{equation}
Here $\det(I-A)_{L^2(t,\infty)}$ is the Fredholm determinant of the trace-class operator
on $L^2(t,\infty)$ with integral kernel $A(x,y)$.

We can now state our result on the edge statistics.

\begin{theorem}\label{thm1.2}
Let $W$ be a Gaussian divisible Hermitian Wigner matrix, (\ref{1.1}), wth finite fourth moments, 
i.e. there is a constant $K<\infty$ independent of $n$ such that
\begin{equation}\label{1.10}
\max_{1\le i\le j\le n}\mathbb{E}[|x_{ij}|^4]\le K.
\end{equation}
Let $\{\lambda_j\}$ be the eigenvalues of $\sqrt{n}W$, and let
\begin{equation}
\gamma=\sqrt{1+4\kappa},\quad\delta=\frac 12 \sqrt{1+4\kappa}.
\notag
\end{equation}
Then,
\begin{equation}\label{1.11}
\lim_{n\to\infty}\mathbb{E}\left[\exp(-\sum_{j=1}^n\psi((\lambda_j-\gamma n)/\delta n^{1/3}))\right]
=\mathbb{E}_{\text{{\rm Airy}}}\left[\exp(-\sum_j\psi(\mu_j))\right]
\end{equation}
for all $\psi\in C^+_c(\mathbb{R})$. Furthermore, if $\lambda_{\max}=\max_{1\le j\le n}\lambda_j$, then
\begin{equation}\label{1.12}
\lim_{n\to\infty}\mathbb{P}[(\lambda_{\max}-\gamma n)/\delta n^{1/3}\le t]=F_{\text{{\rm TW}}}(t),
\end{equation}
for all $t\in\mathbb{R}$.
\end{theorem}

The theorem will be proved in section \ref{sect3.2}.

\begin{remark} \rm When we have two but not four moments we have asymptotically the semi-circle law,
the local eigenvalue statistics in the bulk is given by the sine-kernel point process, but the local
eigenvalue statistics around the largest eigenvalue, which lies outside the semi-circle, is given by
a Poisson process. It would be interesting to investigate the change in statistics as we move towards the
edge. In terms of eigenvectors we should move from localized eigenvectors to de-localized eigenvectors.
This problem is perhaps even more interesting when we have heavy-tailed distributions with unbounded 
variance. The global eigenvalue distribution is then no longer given by the semi-circle law and the scaling
is different, \cite{BG}. See \cite{CB} for a discussion. It is possible that the methods of the present paper
could be extended to yield e.g. the sine-kernel point process close to the origin in this case also. 
This would probably require an 
improvement of the estimate (\ref{2.27}), which still holds, but is not good enough.\it
\end{remark}

As mentioned in the introduction Tao and Vu have recently extended the four-moment theorem to the edge,
but since they compared with GUE they had to assume vanishing third moment. By combining with theorem
\ref{thm1.2} we can see that the third moment condition is not necessary.
We formulate this only for the fluctuations of the largest eigenvalue.

\begin{theorem}\label{thm1.3}
Assume that $M=(m_{ij})$ is an Hermitian Wigner matrix with subexponential decay, i.e. there are
constants $C,C'>0$ such that
\begin{equation}
\mathbb{P}[|m_{ij}|\ge t^C]\le e^{-t}
\notag
\end{equation}
for all $t\ge C'$ and all $1\le i\le j\le n$. Let $\lambda_{\max}$ be the largest eigenvalue of $\sqrt{n}M$,
and assume that the variance $\sigma^2=1$. Then
\begin{equation}\label{1.13}
\lim_{n\to\infty}\mathbb{P}[(\lambda_{\max}-2n)/n^{1/3}\le t]=F_{\text{{\rm TW}}}(t),
\end{equation}
for all $t\in\mathbb{R}$.
\end{theorem}

\begin{proof} We can choose a Gaussian divisible Wigner matrix $M'$ so that the moments
of $M$ and $M'$ match up to order three, see \cite{TV1}. The result then follows from
(\ref{1.12}) and theorem \cite{TV2}, theorem 1.13; compare the proof of theorem 1.16 in \cite{TV2}.
\end{proof}

\section{Bulk universality}\label{sect2}

\subsection{Convergence to the sine kernel point process}\label{sect2.1}

Consider $n$ Brownian motions $x_1(t),\dots,x_n(t)$ on $\mathbb{R}$ starting at 
$\nu_1,\dots,\nu_n$ and conditioned never to intersect. The random positions at time 
$S$ then form a determinantal point process with correlation kernel
\begin{equation}\label{2.1}
K_{n,S}^{\nu}(u,v)=\frac 1{(2\pi i)^2S}\int_{\gamma_L}\,dz\int_{\Gamma_M}\,dw
e^{(w^2-2vw-z^2+2uz)/2S}\frac 1{w-z}\prod_{j=1}^n\frac {w-\nu_j}{z-\nu_j},
\end{equation}
where $\nu=\{\nu_j\}_{j=1}^n$, $\gamma_L$ is the contour given by the positively 
oriented rectangle with corners at $\pm L\pm i$ and $\Gamma_M$ is the contour given by
$s\to M+is$, with $M\ge L$, see \cite{JoUniv}. Here $L$ is chosen so large that all 
the points $\nu_j$ lie inside $\gamma_L$.
Let $\mathbb{E}_{\nu}$ denote the expectation with respect to the family of non-intersecting
Brownian motions, and let $\phi\in C_c(\mathbb{R})$ satisfy $0\le \phi\le 1$. Then,
\begin{equation}\label{2.2}
\mathbb{E}^{\nu}[\prod_{j=1}^n(1-\phi(x_j(S)))]=
\det(I-\phi^{1/2}K_{n,S}^{\nu}\phi^{1/2}),
\end{equation}
where the right hand side is a Fredholm determinant on $L^2(\mathbb{R})$ with respect
to the finite rank kernel $\phi^{1/2}K_{n,S}^{\nu}\phi^{1/2}$. 

This is useful for studying Gaussian divisible Wigner matrices because of the
following fact. Let $\mathbb{E}_X$ denote the expectation with respect to the Wigner matrix $X$ and let
$y(X)=\{y_j(X)\}_{j=1}^n$ be the eigenvalues of $\sqrt{n}X$. Furthermore let $\mathbb{E}_W$ denote
the expectation with respect to the Gaussian divisible Wigner matrix $W$, (\ref{1.1}).
Then, \cite{JoUniv}, for $\psi\in C_c^+(\mathbb{R})$,
\begin{equation}\label{2.19}
\mathbb{E}_W\left[\exp(-\sum_{j=1}^n\psi(\lambda_j-d_n))\right]=\mathbb{E}_X
\left[\mathbb{E}^{y(X)}\left[\exp(-\sum_{j=1}^n\psi(x_j(S_n)-d_n))\right]\right],
\end{equation}
where $\{\lambda_j\}$ are the eigenvalues of $W$ and $S_n=\kappa n$. To use this formula we need good control
of the kernel $K_{n,S}^\nu$ given by (\ref{2.1}) for all $\nu=y(X)$ except a set whose
probability is negligible.

Define, for a given set $\nu$ and positive number $S$ 
\begin{equation}\label{2.3}
B_{n,S}=\{\nu\,;\,\text{there is a $b>0$ such that $\sum_{j=1}^n\frac {S}{\nu_j^2+b^2S^2}=1$}\}.
\end{equation}
Hence, if $\nu\in B_{n,S}$, there is a unique $b=b(\nu)$ such that
\begin{equation}\label{2.4}
\sum_{j=1}^n\frac {S}{\nu_j^2+b^2S^2}=1.
\end{equation}
Furthermore, define for $v\in B_{n,S}$,
\begin{equation}\label{2.5}
D(\nu)=\sum_{j=1}^n\frac {\nu_j}{\nu_j^2+b^2S^2},
\end{equation}
and
\begin{equation}\label{2.6}
A(\nu)=\sum_{j=1}^n\frac {S^3b^2}{(\nu_j^2+b^2S^2)^2}.
\end{equation}
(We suppress the dependence on $n$ and $S$ in the notation for these quantities.)

We have the following approximation theorem.
\begin{theorem}\label{thm2.1} If $\nu\in B_{n,S}$ there is a numerical constant $C$ such that
\begin{equation}\label{2.7}
\left|K_{n,S}^{\nu}(u-SD(\nu), v-SD(\nu))-\frac{\sin b(\nu)(u-v)}{\pi(u-v)}\right|
\le\frac{C}{\sqrt{SA(\nu)}}e^{3u^2/A(\nu)}.
\end{equation}
\end{theorem}

We postpone the proof to section \ref{sect2.3}.
The proof is based, as can be expected, on an asymptotic analysis of the
integral formula (\ref{2.1}). The important point is that the analysis can be done in
such a way that the dependence on $\nu$ only enters in a few quantities, $b(\nu)$, $D(\nu)$ and
$A(\nu)$.

Assume now that we have a probability measure $\mathbb{P}_{\nu}$ with expectation 
$\mathbb{E}_{\nu}$ on the point configurations $\nu=\{\nu_j\}$. We can then define a point 
process $\mu=\{\mu_j\}_{j=1}^n$ on $\mathbb{R}$ depending on $S$ by
\begin{equation}\label{2.8}
\mathbb{E}_{n,S}\left[\prod_{j=1}^n(1-\phi(\mu_j))\right]=\mathbb{E}_{\nu}\left[\mathbb{E}^{\nu}
\left[\prod_{j=1}^n(1-\phi(x_j(S)))\right]\right]
\end{equation}
for every $\phi\in C_c(\mathbb{R})$ with $0\le\phi\le 1$.

We now have the following theorem on convergence to the sine kernel point process defined by
(\ref{1.4}).
\begin{theorem}\label{thm2.2}
Let $\alpha_n$, $\beta_n$, $\delta_n$, $\omega_n$ and $S_n$ be 
sequences such that $S_n>0$, $\omega_n\to\infty$, $\omega_n/\log (S_n\alpha_n)\to 0$ and 
$\beta_n\to\beta>0$ as $n\to\infty$. Define
\begin{equation}\label{2.11}
C_n=\{\nu\in B_{n,S_n}\,;\,A(\nu)\ge\alpha_n,|b(\nu)-\beta_n|\le1/\omega_n,
|D(\nu)-\delta_n|\le\sqrt{\omega_n\alpha_n/S_n}\}.
\end{equation}
Assume that
\begin{equation}\label{2.12}
\lim_{n\to\infty}\mathbb{P}_{\nu}[C_n]=1.
\end{equation}
Then,
\begin{equation}\label{2.13}
\lim_{n\to\infty}\mathbb{E}_{n,S_n}\left[\exp(-\sum_{j=1}^n\psi(\mu_j+S_n\delta_n))\right]=
\mathbb{E}^\beta_{\text{sine}}\left[\exp(\sum_j\psi(\mu_j))\right]
\end{equation}
for every $\psi\in C^+_c(\mathbb{R})$.
\end{theorem}
\begin{proof}
It is clear from (\ref{2.12}) and (\ref{2.8}) that it is enough to prove that
\begin{equation}\label{2.14}
\lim_{n\to\infty}\mathbb{E}_{\nu}\left[1_{C_n}\mathbb{E}^{\nu}
\left[\exp(-\sum_{j=1}^n\psi(x_j(S_n)+S_n\delta_n))\right]\right]=
\mathbb{E}^\beta_{\text{sine}}\left[\exp(\sum_j\psi(\mu_j))\right].
\end{equation}
Here $1_A$ denotes the indicator function for the event $A$.
Write $\phi=1-e^{-\psi}$. Consider a fixed $\nu\in C_n$ and write
\begin{equation}
L_n^{\nu}(u,v)=K_{n,S_n}^{\nu}(u-S_nD(\nu),v-S_nD(\nu))
\notag
\end{equation}
and
\begin{equation}
\phi_n(u)=\phi(u+S_n\delta_n-S_nD(\nu)).
\notag
\end{equation}
It follows from (\ref{2.7}) that 
\begin{align}\label{2.15}
&|\phi_n^{1/2}(u)L_n^{\nu}(u,v)\phi_n^{1/2}(v)-\phi_n^{1/2}(u)K^{b(\nu)}_{\text{sine}}
\phi_n^{1/2}(v)|
\notag\\
&\le\frac{C}{\sqrt{S_nA(\nu)}}e^{Cu^2/S_nA(\nu)}\phi_n^{1/2}(u)\phi_n^{1/2}(v).
\end{align}
There is a constant $C$ such that $\phi_n(u)=0$ if $|u|\ge S_n|D(\nu)-\delta_n|+C$. Hence, 
$\phi_n(u)=0$ if $|u|\ge 2\sqrt{\omega_n\alpha_nS_n}$ for $n$ large since $\nu\in C_n$. If
$|u|\le 2\sqrt{\omega_n\alpha_nS_n}$, then
\begin{equation}
\frac{C}{\sqrt{S_nA(\nu)}}e^{Cu^2/S_nA(\nu)}\le \frac{C}{\sqrt{S_nA(\nu)}}e^{C\omega_n}
\le\frac{C}{(S_n\alpha_n)^{1/4}}
\notag
\end{equation}
for $n$ large, since $\omega_n/\log(S_n\alpha_n)\to 0$ as $n\to\infty$. Thus, by (\ref{2.15})
\begin{equation}
|\phi_n^{1/2}(u)L_n^{\nu}(u,v)\phi_n^{1/2}(v)-\phi_n^{1/2}(u)K^{b(\nu)}_{\text{sine}}
\phi_n^{1/2}(v)|\le
\frac{C}{(S_n\alpha_n)^{1/4}}\phi_n^{1/2}(u)\phi_n^{1/2}(v)
\notag
\end{equation}
for all $u,v$.
For a given $\epsilon>0$ we thus have
\begin{equation}\label{2.17}
|\phi_n^{1/2}(u)L_n^{\nu}(u,v)\phi_n^{1/2}(v)-\phi_n^{1/2}(u)K^{\beta}_{\text{sine}}
\phi_n^{1/2}(v)|\le
\epsilon\phi_n^{1/2}(u)\phi_n^{1/2}(v)
\end{equation}
for all sufficiently large $n$ uniformly in $\nu\in C_n$, since $|b(\nu)-\beta_n|\le 1/\omega_n$
and $\beta_n\to\beta$ as $n\to\infty$.

If $A$ is an operator on $L^2(\mathbb{R})$ with integral kernel $A(x,y)$ then the Hilbert-Schmidt
norm of $A$ is give by $||A||_2^2=\int_{\mathbb{R}^2}|A(x,y)|^2dxdy$. We now use the following lemma.

\begin{lemma}\label{lem2.3} If $A$ and $B$ are trace class operators on $L^2(\mathbb{R})$ then
\begin{align}\label{2.18}
&|\det(I-A)-\det(I-B)|
\notag\\
&\le ||A-B||_2 e^{-\Tr A+(||A-B||_2+2||B||_2+1)^2/2}+
e^{(||B||_2+1)^2/2-\Tr B}(e^{-(\Tr A-\Tr B)}-1).
\end{align}
\end{lemma}

The lemma is proved in section \ref{sect3.4}.

It follows from (\ref{2.2}) and a translation of variables that
\begin{equation}\label{2.18:2}
\mathbb{E}^{\nu}\left[\exp(-\sum_{j=1}\psi(x_j(S_n)+S_n\delta_n))\right]=
\det(I-\phi_n^{1/2}L_n^\nu\phi_n^{1/2}).
\end{equation}
Using (\ref{2.17}), (\ref{2.18}) and the fact that the sine kernel is translation invariant
it is now straightforward to see that
\begin{equation}
|\det(I-\phi_n^{1/2}L_n^{\nu}\phi_n^{1/2})-\det(I-\phi^{1/2}K^{\beta}_{\text{sine}}\phi^{1/2})|\to 0
\notag
\end{equation}
uniformly for $\nu\in C_n$ as $n\to\infty$ . This completes the proof by (\ref{2.14}) and 
(\ref{2.18:2}).
\end{proof}

\subsection{Proof of bulk universality}\label{sect2.2}
In this section we will prove theorem \ref{thm1.1} on bulk universality for Gaussian divisible 
Hermitian Wigner matrices with finite second moment using the convergence therem \ref{thm2.2}.
Define
\begin{equation}\label{2.20}
m_n(z)=\frac 1n\sum_{j=1}^n\frac 1{y_j-z}=\frac 1n\Tr(X/\sqrt{n}-z)^{-1}
\end{equation}
for $\im z\neq 0$. Then
\begin{equation}\label{2.21}
\mathbb{E}_X[m_n(z)]\to m(z)=-2z+\sqrt{z^2-1}
\end{equation}
as $n\to\infty$ (convergence to the semi-circle law), for each $z\in\mathbb{C}$ with
$\im z\neq 0$. Let $\delta+\beta i$, $\beta>0$, be given by
\begin{equation}\label{2.22}
m(d+\kappa(\delta+\beta i))=\delta+\beta i,
\end{equation}
which gives
\begin{equation}\label{2.23}
\delta=-\frac{2d}{1+4\kappa},\quad \beta=\frac 2{1+4\kappa}\sqrt{1+4\kappa-d^2}.
\end{equation}

\begin{lemma}\label{lem2.4} There is a sequence $\delta_n+\beta_n i$, $\beta_n>0$,
such that
\begin{equation}\label{2.24}
\mathbb{E}_X[m_n(d_n/n+\kappa(\delta_n+\beta_n i))]=\delta_n+\beta_n i
\end{equation}
and $\delta_n+\beta_n i\to \delta+\beta i$ as $n\to\infty$.
\end{lemma}

\begin{proof}
Define $g_n(z)=\mathbb{E}_X[m_n(d_n/n+\kappa z)]-z$. Then $g_n$ is analytic in
$\im z>0$. Since
\begin{equation}
|m_n(d_n/n+\kappa z)-m_n(d+\kappa z)|\le\frac{|d_n-d|}{(\kappa \im z)^2}
\notag
\end{equation}
and $d_n/n\to d$ as $n\to\infty$, it follows from (\ref{2.21}) that
$g_n(z)\to g(z)=m(d+\kappa z)-z$ uniformly on compact subsets of $\im z>0$ as $n\to\infty$
(by Montel's therem).
Since $g(\delta +\beta i)=0$ by (\ref{2.22}) it follows by Hurwitz' theorem that there 
is a sequence $\delta_n+\beta_n i$ such that $g_n(\delta_n+\beta_n i)=0$ and 
$\delta_n+\beta_n i\to\delta +\beta i$.
\end{proof}

Set
\begin{equation}\label{2.25}
c_n=d_n/n+\delta_n,\quad \nu_j=y_j-c_nS_n.
\end{equation}

The probability measure on $X$ induces a probabilty measure on $\nu=\{\nu_j\}$ 
that we denote by $\mathbb{P}_{\nu}$. Now, using (\ref{2.1}), we see that
\begin{equation}
K^{y}_{n,S_n}(u+c_nS_n,v+c_nS_n)=e^{((u+c_nS_n)^2-(v+c_nS_n)^2+v^2-u^2)/2S_n}
K^{\nu}_{n,S_n}(u,v)
\notag
\end{equation}
and from this it follows that
\begin{equation}
\mathbb{E}^{y}\left[\exp(-\sum_{j=1}^n\psi(x_j(S_n)-d_n))\right]=
\mathbb{E}^{\nu}\left[\exp(-\sum_{j=1}^n\psi(x_j(S_n)+\delta_nS_n))\right]
\notag
\end{equation}
Hence,
\begin{equation}\label{2.26}
\mathbb{E}_W\left[\exp(-\sum_{j=1}^n\psi(\lambda_j-d_n))\right]=
\mathbb{E}_{\nu}\left[\mathbb{E}^{\nu}\left[\exp(-\sum_{j=1}^n\psi(x_j(S_n)+S_n\delta_n))\right]\right].
\end{equation}
Choose $\alpha_n=\alpha>0$ fixed, to be specified below, $\beta_n$ and $\delta_n$ as in 
lemma \ref{lem2.4}, $S_n=\kappa n$ and $\omega_n=\sqrt{\log n}$. Then theorem \ref{thm1.1}
follows if we can show that $\mathbb{P}_{\nu}[C_n]\to 1$ as $n\to\infty$ with $C_n$ as 
in (\ref{2.11}).

To prove this we will use

\begin{lemma}\label{lem2.5}
For each $z\in\mathbb{C}$ with $\im z\neq 0$ we have the estimate
\begin{equation}\label{2.27}
\mathbb{E}_X[|m_n(z)-\mathbb{E}_X[m_n(z)]|^2\le\frac{2}{n|\im z|^2}.
\end{equation}
\end{lemma}

This is proved in \cite{Bai1}. For convenience we give the proof
in the appendix. 

Define
\begin{equation}
M_n(\tau)=m_n(\kappa c_n+\kappa\tau i)-\mathbb{E}_X[m_n(\kappa c_n+\kappa\tau i)].
\notag
\end{equation}
Note that, by (\ref{2.20}) and (\ref{2.25})
\begin{equation}\label{2.28}
m_n(\kappa c_n+z)=\frac 1n\sum_{j=1}^n\frac 1{\nu_j/n-z}.
\end{equation}
Set
\begin{equation}
V_n=\{\nu\,;\,\text{$|M_n(\tau)|\le\sqrt{\frac{\omega_n}n}$ for $\tau=\beta_n, \beta/2, 2\beta$
and $3\beta$}\}.
\notag
\end{equation}
The result we need now follows from

\begin{lemma}\label{lem2.6}
The following statements hold.
\begin{itemize}
\item[(i)] $\mathbb{P}_{\nu}[V_n]\to 1$ as $n\to\infty$.
\item[(ii)] There is an $\alpha>0$ such that if we choose $\alpha_n=\alpha$ and the other 
sequences as above, then $V_n\subseteq B_{n,S_n}$ and $V_n\cap B_{n,S_n}\subseteq C_n$, if $n$
is large enough.
\end{itemize}
\end{lemma}

\begin{proof}
Let $\tau>0$ be fixed. Then by Chebyshev's inequality and lemma \ref{lem2.5}
\begin{equation}
\mathbb{P}_{\nu}[|M_n(\tau)|>\sqrt{\frac{\omega_n}n}]\le\frac n{\omega_n}
\mathbb{E}_X[|m_n(\kappa c_n+\kappa\tau i)-\mathbb{E}_X[m_n(\kappa c_n+\kappa\tau i)]|^2]
\le\frac{2}{\omega_n\kappa^2\tau^2}\to 0,
\notag
\end{equation}
as $n\to\infty$. We can apply this to $\tau=\beta_n,\beta/2,2\beta$ and $3\beta$ noting
that $\beta_n\ge\beta/2$ if $n$ is large enough. This proves (i).

Note that
\begin{equation}\label{2.29}
\re m_n(\kappa c_n+\kappa\tau i)=\sum_{j=1}^n\frac{\nu_j}{\nu_j^2+\tau^2S_n^2}
\end{equation}
and
\begin{equation}\label{2.30}
\im m_n(\kappa c_n+\kappa\tau i)=\sum_{j=1}^n\frac{S_n\tau}{\nu_j^2+\tau^2S_n^2}
\end{equation}
Furthermore,
\begin{equation}
h(\tau)=\frac 1{\tau}\im m(x+i\tau)=\frac 2{\pi}\int_{-1}^1\frac{\sqrt{1-t^2}}{(t-x)^2+\tau^2}\,dt
\notag
\end{equation}
is strictly decreasing in $\tau$ for each fixed $x$.

Define
\begin{equation}
U_n=\{\nu\,;\,\sum_{j=1}^n\frac{S_n}{\nu_j^2+4\beta^2S_n^2}<1<
\sum_{j=1}^n\frac{S_n}{\nu_j^2+\beta^2S_n^2/4}\}.
\notag
\end{equation}
We want to show that $V_n\subseteq U_n$ if $n$ is large enough. Since $h(\tau)$ is strictly decreasing,
(\ref{2.22}) gives
\begin{equation}
\frac 1{2\beta}\im m(\kappa c+2\kappa\beta i)<1-\epsilon<1=\frac 1{\beta}\im m(\kappa c+\kappa\beta i)
<1+\epsilon<\frac 2{\beta}\im m(\kappa c+\kappa\beta i/2),
\notag
\end{equation}
if we choose $\epsilon$ small enough. Here $c=d/\kappa+\delta=\lim_{n\to\infty} c_n$. It follows from this
and (\ref{2.21}) that 
\begin{equation}
\frac 1{2\beta}|\im\mathbb{E}_X[m_n(\kappa c_n+2\kappa\beta i)]\le 1-\epsilon
<1+\epsilon\le\frac 2{\beta}\im\mathbb{E}_X[m_n(\kappa c_n+\kappa\beta i/2)]
\notag
\end{equation}
for all $n$ large enough. If $\nu\in V_n$ it follows from this that
\begin{equation}
\frac 1{2\beta}\im m_n(\kappa c_n+2\kappa\beta i)\le 1-\epsilon+\sqrt{\frac{\omega_n}n}<1<
1+\epsilon-\sqrt{\frac{\omega_n}n}\le\frac 2{\beta}\im m_n(\kappa c_n+\kappa\beta i/2)
\notag,
\end{equation}
and we see from (\ref{2.30}) that this gives $\nu\in U_n$.

Hence, if $n$ is large enough, then
\begin{equation}\label{2.31}
\beta/2\le b(\nu)\le 2\beta
\end{equation}
for all $\nu\in V_n$.
Let $\nu\in V_n$. Then, using (\ref{2.31}), we see that
\begin{align}
&A(\nu)=\sum_{j=1}^n\frac{S_n^3b^2}{(\nu_j^2+b^2S_n^2)^2}\ge
\frac 14\sum_{j=1}^n\frac{S_n^3\beta^2}{(\nu_j^2+4\beta^2S_n^2)^2}
\notag\\
&\ge \frac{S_n}{20}\sum_{j=1}^n\frac{5S_n^2\beta^2}{(\nu_j^2+4\beta^2S_n^2)(\nu_j^2+9\beta^2S_n^2)}
=\frac 1{20}\left(\sum_{j=1}^n\frac{S_n}{\nu_j^2+4\beta^2S_n^2}-
\sum_{j=1}^n\frac{S_n}{\nu_j^2+9\beta^2S_n^2}\right).
\notag
\end{align}
By (\ref{2.21}), (\ref{2.30}) and the fact that $\nu\in V_n$ it follows from this that
\begin{align}
A(\nu)&\ge\frac 1{20}\left(\frac 1{2\beta}\im M_n(2\beta)-\frac 1{3\beta}\im M_n(3\beta)\right)
\notag\\
&+\frac 1{20}\left(\frac1{2\beta}\im\mathbb{E}_X[m_n(\kappa c_n+2\kappa\beta i)]-
\frac1{3\beta}\im\mathbb{E}_X[m_n(\kappa c_n+3\kappa\beta i)]\right)
\notag\\
&\ge\frac 1{40}\left(\frac 1{2\beta}\im m(\kappa c+2\kappa\beta i)-
\frac 1{3\beta}\im m(\kappa c+3\kappa\beta i)\right)\doteq\alpha>0,
\notag
\end{align}
for large $n$.

Next, we will show that, if $n$ is large enough,
\begin{equation}\label{2.32}
|b(\nu)-\beta_n|\le C\sqrt{\frac{\omega_n}n}\le \frac 1{\omega_n}
\end{equation}
for all $\nu\in V_n$. It follows from (\ref{2.4}), (\ref{2.24}), (\ref{2.30}) and $\nu\in V_n$, that
\begin{equation}
\left|\sum_{j=1}^n\frac{S_n}{\nu_j^2+\beta_n^2S_n^2}-\sum_{j=1}^n\frac{S_n}{\nu_j^2+b^2S_n^2}
\right|\le \sqrt{\frac{\omega_n}n},
\notag
\end{equation}
which implies
\begin{equation}
|b^2-\beta_n|\sum_{j=1}^n\frac{S_n^3}{(\nu_j^2+
\beta_n^2S_n^2)(\nu_j^2+b^2S_n^2)}\le \sqrt{\frac{\omega_n}n}.
\notag
\end{equation}
Now,
\begin{equation}
\sum_{j=1}^n\frac{S_n^3}{(\nu_j^2+\beta_n^2S_n^2)(\nu_j^2+b^2S_n^2)}\ge
\frac 1{\beta^2}
\sum_{j=1}^n\frac{S_n^3\beta^2}{(\nu_j^2+4\beta^2S_n^2)(\nu_j^2+9\beta^2S_n^2)}\ge\frac{20\alpha}{\beta^2}
\notag
\end{equation}
by the previous argument, since $\beta_n\le 3\beta$ for large $n$ and $b\le 2\beta$ by (\ref{3.31}).
Consequently,
\begin{equation}
|b-\beta_n|\le\frac{\beta}{20\alpha} \sqrt{\frac{\omega_n}n},
\notag
\end{equation}
since $b+\beta_n\ge \beta/2+\beta_n\ge \beta$ for large $n$. This proves (\ref{2.32}).

It remains to show that 
\begin{equation}\label{2.33}
|D(\nu)-\delta_n|\le C\sqrt{\frac{\omega_n}n}
\end{equation}
for all $\nu\in$ and large $n$. It follows from (\ref{2.24}) and (\ref{2.29}) that
\begin{align}
D(\nu)-\delta_n&=\re m_n(\kappa c_n+\kappa bi)-\re m_n(\kappa c_n+\kappa \beta_ni)
\notag\\
&+\re m_n(\kappa c_n+\kappa \beta_ni)-\re \mathbb{E}_X[m_n(\kappa c_n+\kappa \beta_ni)].
\notag
\end{align}
We can use (\ref{2.31}) and (\ref{2.32}) to show that
\begin{equation}
|\re m_n(\kappa c_n+\kappa bi)-\re m_n(\kappa c_n+\kappa \beta_ni)|\le\frac{\kappa|b-\beta_n|}
{\kappa^2\beta_n b}\le C\sqrt{\frac{\omega_n}n}.
\notag
\end{equation}
Furthermore, the definition of $V_n$ gives
\begin{equation}
|\re m_n(\kappa c_n+\kappa \beta_ni)-\re \mathbb{E}_X[m_n(\kappa c_n+\kappa \beta_ni)]|
\le \sqrt{\frac{\omega_n}n}
\notag
\end{equation}
for all $\nu\in V_n$.

This proves (ii) of lemma \ref{lem2.6}.
\end{proof}

\subsection{Proof of the approximation theorem}\label{sect2.3}

In this section we will prove the convergence theorem \ref{thm2.1}. A change of variables gives
\begin{equation}\label{2.34}
K_{n,S}^{\nu}(u,v)=\frac 1{(2\pi i)^2}\int_{\gamma_L}\,dz\int_{\Gamma_M}\,dw
e^{S(w^2-z^2)/2+uz-vw}\frac 1{w-z}\prod_{j=1}^n\frac{Sw-\nu_j}{Sz-\nu_j},
\end{equation}
where $\gamma_L$ is the positively oriented rectangle with corners at $\pm L\pm bi$, $|\nu_j|<L$ for all $j$
and $M>L$. Set
\begin{equation}\label{2.33:2}
f(z)=\frac {z^2}2+D(\nu)z+\frac 1S\sum_{j=1}^n\log (Sz-\nu_j)
\end{equation}
and
\begin{equation}\label{2.35}
\tilde{K}_{n,S}^{\nu}(u,v)=\frac 1{(2\pi i)^2}\int_{\gamma}\,dz\int_{\Gamma}\,dw
e^{uz-vw}\frac 1{w-z}e^{S(f(w)-f(z))},
\end{equation}
where $\gamma=\gamma_++\gamma_-$ and $\gamma_{\pm}:t\to\mp t\pm ib$, $t\in\mathbb{R}$, and
$\Gamma=\Gamma_0: s\to is$, $s\in\mathbb{R}$. If we move $\Gamma_M$ to $\Gamma_0$ and let 
$L\to\infty$ it follows from the residue theorem that
\begin{equation}\label{2.36}
K_{n,S}^{\nu}(u-SD(\nu),v-SD(\nu))-\frac{\sin b(u-v)}{\pi(u-v)}=\tilde{K}_{n,S}^{\nu}(u,v).
\end{equation}
Hence, theorem \ref{thm2.1}, follows from
\begin{equation}\label{2.37}
|\tilde{K}_{n,S}^{\nu}(u,v)|\le\frac{C}{\sqrt{SA(\nu)}}e^{3u^2/SA(\nu)}
\end{equation}
for all $v\in B_{n,S}$. In order to prove this inequality we have to choose the right contours in
(\ref{2.35}). The following computation motivates the choice of contours.

Let $z(t)=x(t)+iy(t)$ and set $g(t)=\re f(z(t))$. Then, using (\ref{2.4}) and (\ref{2.5}) we see
that
\begin{equation}\label{2.38}
g'=\sum_{j=1}^n\left[\frac{S(xx'-yy')+x\nu_j}{\nu_j^2+b^2S^2}+
\frac{S(xx'+yy')-x'\nu_j}{(Sx-\nu_j)^2+S^2y^2}\right].
\end{equation}
If we write the sum of the two fractions in (\ref{2.38}) as one fraction the numerator becomes
\begin{equation}
S^2[-x^2x'+2xyy'+y^2x'-b^2x']\nu_j
+S^3[(xx'-yy')(x^2+y^2)+b^2(xx'+yy')].
\notag
\end{equation}
We try to choose $z(t)$ so that the expression in the numerator is independent of $\nu_j$.
This gives
\begin{equation}
\frac{d}{dt}[-\frac 13x^3+y^2x-b^2x]=0
\notag
\end{equation}
or
\begin{equation}
x[-\frac 13x^2+y^2-b^2]=C.
\notag
\end{equation}
If $x(0)=0$, $y(0)=\pm b$ we get $C=0$ and two possibilities $z(t)=i(t\pm b)$ or
$z(t)=t\pm i\sqrt{t^2/3+b^2}$. 

If we take $z(t)=i(t\pm b)$ we get
\begin{equation}\label{2.39}
\frac{d}{dt}\re f(z(t))=-St\sum_{j=1}^n\frac{S^2(t\pm b)(t\pm 2b)}{(\nu_j^2+b^2S^2)(\nu_j^2+(t+b)^2S^2)}.
\end{equation}
If instead we take $z(t)=t\pm i\sqrt{t^2/3+b^2}$ we obtain
\begin{equation}\label{2.40}
\frac{d}{dt}\re f(z(t))=St\sum_{j=1}^n\frac{8S^2t^2/9+2b^2S^2}{(\nu_j^2+b^2S^2)((St-\nu_j)^2+(t^2/3+b^2)S^2)}.
\end{equation}
Using this result we can prove

\begin{lemma}\label{lem2.7}
Let $w_{\pm}(s)=i(s\pm ib)$ and $z_{\pm}(t)=t\pm i\sqrt{t^2/3+b^2}$. Assume that $\nu\in B_{n,S}$. 
\begin{itemize}
\item[(i)] If $\pm s+b\ge 0$, then
\begin{equation}\label{2.41}
\re (f(w_{\pm}(s))-f(\pm bi))\le-\frac 16A(\nu)s^2.
\end{equation}
\item[(ii)] For each $t\in\mathbb{R}$,
\begin{equation}\label{2.42}
\re (f(\pm bi)-f(z_{\pm}(t)))\le-\frac 16A(\nu)t^2.
\end{equation}
\end{itemize}
\end{lemma}

\begin{proof}
We see that, for $-b\le s\le 0$,
\begin{align}
&\re (f(w_{+}(s))-f(bi))=S^3\int_s^0 t\sum_{j=1}^n\frac{(t+b)(t+2b)}{(\nu_j^2+b^2S^2)(\nu_j^2+(b+t)^2S^2)}\,dt
\notag\\
&\le S^3\int_s^0 t\sum_{j=1}^n\frac{(b+t)b}{(\nu_j^2+b^2S^2)^2}\,dt
\notag\\
&=\frac{A(\nu)}b \left(-\frac{s^2}3\right)(\frac 32b+s)\le-\frac {A(\nu)}6 s^2.
\notag
\end{align}
If $s\ge 0$, we get
\begin{align}
&\re (f(w_{+}(s))-f(bi))=S^3\int_0^s t\sum_{j=1}^n\frac{(t+b)(t+2b)}{(\nu_j^2+b^2S^2)(\nu_j^2+(b+t)^2S^2)}\,dt
\notag\\
&\le -\int_0^s t\sum_{j=1}^n\frac{S^3(t+b)^2}{(\nu_j^2+b^2S^2)(\nu_j^2+(b+t)^2S^2)}\,dt.
\notag
\end{align}
If we use the fact that $x\to x^2(\nu^2+x^2)^{-1}$ is increasing in $x\ge b$, 
we see that the last expression
is
\begin{equation}
\le-A(\nu)\int_0^s t\,dt=-\frac 12A(\nu) s^2.
\notag
\end{equation}
The contour $w_-(s)$ is treated analogously. This proves (i) in the lemma.

Now, for $t\ge 0$,
\begin{align}
&\re (f(z_{+}(s))-f(bi))=S\int_0^t\tau\sum_{j=1}^n\frac{8S^2\tau^2/9+2b^2S^2}{(\nu_j^2+b^2S^2)
((S\tau-\nu_j)^2+S^2(\tau^2/3+b^2))}\,d\tau
\notag\\
&\ge S\int_0^t\tau\sum_{j=1}^n\frac{8S^2\tau^2/9+2b^2S^2}{(\nu_j^2+b^2S^2)(2\nu_j^2+7S^2\tau^2/3+b^2S^2)}.
\notag
\end{align}
It is easy to see that
\begin{equation}
\frac{8S^2\tau^2/9+2b^2S^2}{2\nu_j^2+7S^2\tau^2/3+b^2S^2}\ge\frac 13\frac{S^2b^2}{\nu_j^2+b^2S^2}
\notag
\end{equation}
and hence we obtain (\ref{2.42}) for $z_+(t)$ and $t\ge 0$. The argument for $t\le 0$ and the argument
for $z_-(t)$ are similar.
\end{proof}

We can now prove the estimate (\ref{2.37}). Let $-\gamma_+$ be given by $z_+(t)$, $t\in\mathbb{R}$,
$\gamma_-$ by $z_-(t)$, $t\in\mathbb{R}$, $\Gamma_+$ by $w_+(s)$, $s\ge -b$ and $\Gamma_-$ by
$w_-(s)$, $s\le b$. Then,
\begin{equation}
\tilde{K}_{n,S}^{\nu}(u,v)=\frac 1{(2\pi i)^2}\int_{\gamma_++\gamma_-}\,dz\int_{\Gamma_++\Gamma_-}\,dw
e^{uz-vw}\frac 1{w-z}e^{S(f(w)-f(z))},
\notag
\end{equation}
Consider the case when $z$ lies on $\gamma_+$ and $w$ on $\Gamma_+$. The other cases are similar.
Now, by lemma \ref{lem2.7}
\begin{align}\label{2.43}
&\left|\frac 1{(2\pi i)^2}\int_{\gamma_+}\,dz\int_{\Gamma_+}\,dw
e^{uz-vw}\frac 1{w-z}e^{S(f(w)-f(z))}\right|
\notag\\
&\le\frac 1{4\pi^2}\int_{-\infty}^\infty\,dt\int_{-b}^\infty\,ds\frac{e^{ut}}
{\sqrt{t^2+(b+s-\sqrt{t^2/3+b^2})^2}}e^{-SA(\nu)(s^2+t^2)/6}.
\end{align}
Since $t^2+(b+s-\sqrt{t^2/3+b^2})^2\ge (t^2+s^2)/3$, we see that the expression in the right hand side of 
(\ref{2.43}) is
\begin{equation}
\le\frac{\sqrt{2}}{4\pi^2}\int_{\mathbb{R}^2}\frac{e^{ut}}{\sqrt{t^2+s^2}}e^{-SA(\nu)(s^2+t^2)/6}
\le\frac{C}{SA(\nu)}e^{3u^2/A(\nu)},
\notag
\end{equation}
where $C$ is a numerical constant. This completes the proof of the approximation theorem.

\section{Edge universality}\label{sect3}

\subsection{Convergence to the Airy kernel point process}\label{sect3.1}
Let $\nu=\{\nu_j\}_{j=1}^n\subseteq\mathbb{R}$ and $S>0$ be given. We can then choose $b=b(\nu)$ so that 
$bS>\max\nu_j$ and 
\begin{equation}\label{3.1}
\sum_{j=1}^n\frac{S}{(bS-\nu_j)^2}=1.
\end{equation}
Define $a=a(\nu)$ and $d=d(\nu)$ by
\begin{equation}\label{3.2}
a=b+\sum_{j=1}^n\frac{1}{bS-\nu_j}
\end{equation}
and
\begin{equation}\label{3.3}
d=\left(\sum_{j=1}^n\frac{S^2}{(bS-\nu_j)^3}\right)^{1/3}.
\end{equation}
Let $0<\alpha_0<\beta_0$ be given and define
\begin{equation}\label{3.4}
F_n=\{\nu\,;\,\text{$\alpha_0\le b-\nu_j/S\le\beta_0$ for $1\le j\le n$}\} .
\end{equation}
We then have the following estimate and limit result for the correlation kernel given by
(\ref{2.1}).

\begin{theorem}\label{thm3.1}
There are constants $C$ and $S_0$ depending only on $\alpha_0$, $\beta_0$ so that
\begin{equation}\label{3.5}
dS^{1/3}K_{n,S}^{\nu}(aS+\xi dS^{1/3},aS+\xi dS^{1/3})\le Ce^{-\xi}
\end{equation}
for all $\nu\in F_n$, $\xi\ge 0$ and $S\ge S_0$.Furthermore, if $S=\kappa n$, with $\kappa>0$ fixed, then
\begin{equation}\label{3.6}
\lim_{n\to\infty}dS_n^{1/3}e^{(\eta-\xi)dS_n^{1/3}}K_{n,S}^{\nu}(aS_n+\xi dS_n^{1/3},aS_n+\eta dS_n^{1/3})
=A(\xi,\eta)
\end{equation}
uniformly for $\nu\in F_n$ and $\xi,\eta$ in a compact subset of $\mathbb{R}$. Here $A(\xi,\eta)$ is the 
Airy kernel (\ref{1.7}).
\end{theorem}

The theorem will be proved in section \ref{sect3.3}.

Let $\gamma_n$ and $\epsilon_n$ be given sequences of positive numbers, where $\epsilon_n\to 0$ as $n\to\infty$.
Take $S_n=\kappa n$, $\kappa>0$, let $\delta>0$ be given and define
\begin{equation}\label{3.7}
G_n=\{\nu\in F_n\,;\,\left|\frac{a(\nu)-\gamma_n}{n^{1/3}}\right|\le\epsilon_n,
|d(\nu)-\delta|\le\epsilon_n\}.
\end{equation}
Let $\mathbb{P}_{\nu}$ be a probability measure on point configurations $\nu=\{\nu_j\}_{j=1}^n$ in 
$\mathbb{R}$, and let $\mathbb{E}_{n,S_n}$ be the expectation for the point process $\mu=\{\mu_j\}_{j=1}^n$
on $\mathbb{R}$ defined by (\ref{2.8}).

\begin{theorem}\label{thm3.2} Assume that there is a choice of $\alpha_0,\beta_0,\gamma_n,\epsilon_n,\delta$,
where $\epsilon_n\to 0$ as $n\to\infty$, so that
\begin{equation}\label{3.8}
\lim_{n\to\infty}\mathbb{P}_{\nu}[G_n]=1.
\end{equation}
Then, for any $\psi\in C_c^+(\mathbb{R})$,
\begin{equation}\label{3.9}
\lim_{n\to\infty}\mathbb{E}_{n,S_n}\left[\exp(-\sum_{j=1}^n\psi((\mu_j-\gamma_n)/\delta n^{1/3}))\right]
=\mathbb{E}_{\text{Airy}}\left[\exp(-\sum_{j=1}^n\psi(\mu_j))\right].
\end{equation}
Furthermore,
\begin{equation}\label{3.10}
\lim_{n\to\infty}\mathbb{P}_{n,S_n}\left[\frac 1{\delta n^{1/3}}(\max_{1\le j\le n}\mu_j-\gamma_n)\le t\right]
=F_{\text{TW}}(t)
\end{equation}
for each $t\in\mathbb{R}$.
\end{theorem}

\begin{proof}
We see from (\ref{2.8}), with $\phi=1-e^{-\psi}$, and $(\ref{3.8})$ that to prove (\ref{3.9}) it is enough to show
that
\begin{equation}\label{3.11}
\lim_{n\to\infty}\mathbb{E}_{\nu}\left[1_{G_n}\mathbb{E}^{\nu}\left[
\exp(-\sum_{j=1}^n\psi((\mu_j-\gamma_n)/\delta n^{1/3}))\right]\right]
=\mathbb{E}_{\text{Airy}}\left[e^{-\sum_{j=1}^n\psi(\mu_j)}\right].
\end{equation}
Let
\begin{equation}
\tilde{K}_n^{\nu}(\xi,\eta)=dS_n^{1/3}e^{(\xi-\eta)dS_n^{1/3}}K_{n,S_n}^{\nu}(aS_n+\xi dS_n^{1/3},
aS_n+\eta dS_n^{1/3})
\notag
\end{equation}
and
\begin{equation}
\tilde{\phi}_n(\xi)=\phi(\xi d/\delta+(aS_n-\gamma_n)/\delta n^{1/3}).
\notag
\end{equation}
Then,
\begin{equation}\label{3.12}
\mathbb{E}^{\nu}\left[
\exp(-\sum_{j=1}^n\psi((\mu_j-\gamma_n)/\delta n^{1/3}))\right]=\det(I-\tilde{\phi}_n^{1/2}\tilde{K}_n^{\nu}
\tilde{\phi}_n^{1/2}).
\end{equation}
If $\nu\in G_n$ there is a constant $C$, depending on $\phi$, such that
\begin{equation}\label{3.13}
|\tilde{\phi}_n(\xi)-\phi(\xi)|\le C\epsilon_n.
\end{equation}
If we use (\ref{3.6}), (\ref{3.13}) and the fact that $\phi$ has compact support, we can use lemma 
\ref{lem2.3} to show that
\begin{equation}\label{3.14}
\lim_{n\to\infty}\det(I-\tilde{\phi}_n^{1/2}\tilde{K}_n^{\nu}
\tilde{\phi}_n^{1/2})=\det(I-\phi^{1/2}A\phi^{1/2}),
\end{equation}
uniformly for $\nu\in G_n$, where $A$ is the Airy kernel, (\ref{1.7}). The limit (\ref{3.9}) now follows from 
(\ref{2.8}), (\ref{3.6}), (\ref{3.11}), (\ref{3.12}) and (\ref{3.13}).

It remains to show (\ref{3.10}). Again, from (\ref{3.8}), we see that it is enough to show that
\begin{equation}\label{3.15}
\lim_{n\to\infty}\mathbb{E}_{\nu}[1_{G_n}\mathbb{E}^{\nu}[1_{\#(\gamma_n+\delta tn^{1/3},\infty)=0}]]=
F_{\text{TW}}(t),
\end{equation}
where $\#(x,y)=$ the number of points in $(x,y)$. Take $\tau>t$. Then,
\begin{equation}\label{3.16}
\lim_{n\to\infty}\mathbb{E}_{\nu}[1_{G_n}\mathbb{E}^{\nu}
[1_{\#(\gamma_n+\delta tn^{1/3},\gamma_n+\tau\delta n^{1/3})=0}]]=\det(I-A)_{L^2(t,\tau)}
\end{equation}
follows by an argument analogous to the one above used to prove (\ref{3.9}). Now,
\begin{align}\label{3.17}
\mathbb{E}^{\nu}[1_{\#(\gamma_n+\delta tn^{1/3},\infty)=0}]&=
\mathbb{E}^{\nu}[1_{\#(\gamma_n+\delta tn^{1/3},\gamma_n+\tau\delta n^{1/3})=0}]
\notag\\
&-
\mathbb{E}^{\nu}[1_{\#(\gamma_n+\delta \tau n^{1/3},\infty)\ge 1}
1_{\#(\gamma_n+\delta tn^{1/3},\gamma_n+\tau\delta n^{1/3})=0}].
\end{align}
The second term in the right hand side of (\ref{3.17}) is bounded by
\begin{align}\label{3.18}
&\mathbb{E}^{\nu}[1_{\#(\gamma_n+\delta \tau n^{1/3},\infty)\ge 1}]\le
\mathbb{E}^{\nu}[\#(\gamma_n+\delta \tau n^{1/3},\infty)]
\notag\\
&=\int_{\gamma_n+\delta \tau n^{1/3}}^\infty K_{n,S_n}^{\nu}(x,x)\,dx\le C\int_{(\gamma-aS_n)/dn^{1/3}+
\delta\tau/d}^\infty e^{-\xi}\,d\xi,
\end{align}
where the last inequality follows from (\ref{3.5}) if $\tau$ is sufficiently large, since then
\begin{equation}\label{3.19}
(\gamma-aS_n)/dn^{1/3}+\delta\tau/d\ge -\epsilon_n/\alpha_0+\delta\tau/\beta_0\ge 0.
\end{equation}
Hence, by (\ref{3.16}), (\ref{3.17}), (\ref{3.18}) and (\ref{3.19}),
\begin{align}\label{3.20}
&\limsup_{n\to\infty}\left|\mathbb{E}_{\nu}[1_{G-n}\mathbb{E}^{\nu}[\#(\gamma_n+t\delta tn^{1/3},\infty)]]-
\det(I-A)_{L^2(t,\infty)}\right|
\notag\\
&\le \left|\det(I-A)_{L^2(t,\infty)}-\det(I-A)_{L^2(t,\tau)}\right|+
C\int_{ -\epsilon_n/\alpha_0+\delta\tau/\beta_0}^\infty e^{-\xi}\,d\xi.
\end{align}
If we let $\tau\to\infty$ the right hand side of (\ref{3.20}) goes to zero and we have proved (\ref{3.10}).
\end{proof}

\subsection{Proof of edge universality}\label{sect3.2}

In this section we will prove theorem \ref{thm1.2} on edge universality for Gaussian divisible 
Hermitian Wigner matrices with finite fourth moments.

Let $\nu=y$, where $y=\{y_j\}$ are the eigenvalues of $X$. The expectation $\mathbb{E}_X$ on $X$ induces
an expectation $\mathbb{E}_{\nu}$ on $\nu$. By (\ref{2.19}),
\begin{align}
&\mathbb{E}_{W}\left[\exp(-\sum_{j=1}^n\psi((\lambda_j-\gamma_n)/\delta n^{1/3}))\right]
\notag\\
&=
\mathbb{E}_{\nu}\left[\mathbb{E}^{\nu}\left[
\exp(-\sum_{j=1}^n\psi((x_j(S_n)-\gamma_n)/\delta n^{1/3}))\right]\right].
\notag
\end{align}
By theorem \ref{thm3.2} it is enough to show that there is a choice of $\alpha_0$, $\beta_0$, $\gamma_n$,
$\epsilon_n$ and $\delta$, where $\epsilon_n\to 0$ as $n\to\infty$, so that (\ref{3.8}) holds with 
$G_n$ defined by (\ref{3.7}) and $F_n$ by (\ref{3.4}).

Let $u(x)=\frac 2{\pi}\sqrt{1-x^2}$ be the Wigner semi-circle law with support in $[-1,1]$. We can choose
$b_0>1/\kappa$ so that
\begin{equation}\label{3.21}
\int_{-1}^1\frac{\kappa u(x)}{(b_0\kappa-x)^2}\,dx=1,
\end{equation}
which gives $b_0=(1+2\kappa)(1+4\kappa)^{-1/2}$ by (\ref{2.21}). Let
\begin{equation}
\epsilon=\frac 13 \left(\frac{1+2\kappa}{\sqrt{1+4\kappa}}-1\right),
\notag
\end{equation}
so that $b_0\kappa\ge 1+3\epsilon$. We take $\gamma_n=n\sqrt{1+4\kappa}$ and note that
\begin{equation}\label{3.22}
\gamma_n=n\left(b_0\kappa+\int_{-1}^1\frac{\kappa u(x)}{b_0\kappa-x}\,dx\right).
\end{equation}
Also, we choose $\delta=\frac 12\sqrt{1+4\kappa}$ and note that
\begin{equation}\label{3.23}
\delta^3=\int_{-1}^1\frac{\kappa^3}{(b_0\kappa-x)^3}u(x)\,dx
\end{equation}
Furthermore, we take $\epsilon_n=(\log n)^{-1}$, $\alpha_0=\epsilon/\kappa$ and
$\beta_0=b_0+(1+2\epsilon)/\kappa$.

Define the function $\psi_\beta$ by
\begin{equation}
\psi_\beta(x)=
\begin{cases}
\frac{\kappa}{\kappa\beta-x}, &\text{if $|x|\le 1+\epsilon$} \\
0,                            &\text{if $|x|\ge 1+3\epsilon$,}
\end{cases}
\notag
\end{equation}
and for $1+\epsilon\le |x|\le 1+3\epsilon$ we define $\psi_\beta$ so that it becomes a $C^\infty$ function.

Set
\begin{equation}
H_n'=\{\nu\,;\,\max_{1\le j\le n}|\nu_j|\le 1+\epsilon\}.
\notag
\end{equation}
and
\begin{align}
H_n=&H_n'\cap\left\{\nu\,;\,\left|\sum_{j=1}^n\psi'_{2b_0}(\nu_j/n)-n\int_{-1}^1\psi'_{2b_0}(x)
u(x)\,dx\right|\le n^{1/6}\right.
\notag\\
&\left.\text{and}\,\,
\left|\sum_{j=1}^n\psi^{(j)}_{b_0}(\nu_j/n)-n\int_{-1}^1\psi^{(j)}_{b_0}(x)
u(x)\,dx\right|\le n^{1/6}\,\,\text{for $j=0,1,2$}\right\}.
\notag
\end{align}
We will prove the following lemma below.

\begin{lemma}\label{lem3.3} Let $H_n$ be defined as above. Then,
\begin{equation}\label{3.24}
\lim_{n\to\infty}\mathbb{P}_{\nu}[H_n]=1.
\end{equation}
\end{lemma}

Before we prove the lemma we will use it to show what we want by proving that $H_n\subseteq G_n$.

Let us first show that there is a constant C so that
\begin{equation}\label{3.25}
|b(\nu)-b_0|\le Cn^{-5/6}
\end{equation}
for all $\nu\in H_n$. We see from (\ref{3.1}), (\ref{3.21}) and the definition of $\psi_{b_0}$ that
\begin{align}\label{3.26}
&\left|\int_{-1}^1\psi'_{b_0}(x)u(x)\,dx-\frac 1n\sum_{j=1}^n\psi'_{b_0}(\nu_j/n)\right|
=\left|1-\sum_{j=1}^n\frac{S_n}{(b_0S_n-\nu_j)^2}\right|
\notag\\
&=S_n^2|b-b_0|\sum_{j=1}^n\frac{b_0S_n-\nu_j+bS_n-\nu_j}{(bS_n-\nu_j)^2(b_0S_n-\nu_j)^2}.
\end{align}
We want to show that $b\le 2b_0$. Since $\nu\in H_n$,
\begin{equation}
\frac 1n\sum_{j=1}^n\psi'_{2b_0}(\nu_j/n)\le \int_{-1}^1\psi'_{2b_0}(x)u(x)\,dx+n^{-5/6},
\notag
\end{equation}
which gives
\begin{equation}
\sum_{j=1}^n\frac{S_n}{(2b_0S_n-\nu_j)^2}\le\int_{-1}^1\frac{\kappa u(x)}{(2b_0\kappa-x)^2}\,dx
+n^{-5/6}<\int_{-1}^1\frac{\kappa u(x)}{(b_0\kappa-x)^2}\,dx=1
\notag
\end{equation}
if $n$ is sufficiently large. Hence $b\le 2b_0$. This gives
\begin{align}
&\sum_{j=1}^n\frac{(b_0S_n-\nu_j+bS_n-\nu_j)S_n^2}{(bS_n-\nu_j)^2(b_0S_n-\nu_j)^2}
\ge 2\sum_{j=1}^n\frac{(b_0S_n-\nu_j)^{1/2}(bS_n-\nu_j)^{1/2}S_n^2}{(bS_n-\nu_j)^2(b_0S_n-\nu_j)^2}
\notag\\
&\ge\sum_{j=1}^n\frac{2S_n^2}{(bS_n-\nu_j)^{3/2}(b_0S_n-\nu_j)^{3/2}}\ge
\frac{2\kappa^2}{(2b_0\kappa+1+\epsilon)^{3/2}(b_0\kappa+1+\epsilon)^{3/2}}\doteq c_1,
\notag
\end{align}
since $\nu_j/n\ge -1-\epsilon$ if $\nu\in H_n$. Thus (\ref{3.26}) implies
\begin{equation}
c_1|b-b_0|\le\left|\int_{-1}^1\psi'_{b_0}(x)u(x)\,dx-\frac 1n\sum_{j=1}^n\psi'_{b_0}(\nu_j/n)\right|
\le n^{-5/6}.
\notag
\end{equation}
This proves (\ref{3.25}).

Next, we show that for all $\nu\in H_n$,
\begin{equation}\label{3.27}
\left|\frac{a(\nu)S_n-\gamma_n}{n^{1/3}}\right|\le n^{-1/6}.
\end{equation}
Define,
\begin{equation}\label{3.28}
a_0=b_0+\sum_{j=1}^n\frac 1{b_0S_n-\nu_j}
\end{equation}
an approximate version of (\ref{3.2}). Then, by (\ref{3.1}),
\begin{align}
(a-a_0)S_n&=(b-b_0)S_n\left[1-\sum_{j=1}^n\frac{S_n}{(bS_n-\nu_j)(b_0S_n-\nu_j)}\right]
\notag\\
&=-(b-b_0)^2S_n^2\sum_{j=1}^n\frac{S_n}{(bS_n-\nu_j)^2(b_0S_n-\nu_j)}
\notag
\end{align}
Now,
\begin{equation}
b_0S_n-\nu_j\ge n(b_0\kappa-(1+\epsilon))\ge 2\epsilon n,
\notag
\end{equation}
which gives
\begin{equation}
|a-a_0|S_n\le\frac{\kappa^2 n}{2\epsilon}|b-b_0|^2\sum_{j=1}^n\frac{S_n}{(bS_n-\nu_j)^2}\le Cn^{-2/3}.
\notag
\end{equation}
by (\ref{3.25}) and (\ref{3.1}). From (\ref{3.25}) we also obtain
\begin{align}
|a_0S_n-\gamma_n|&=\left|\sum_{j=1}^n\frac{\kappa n}{b_0\kappa n-\nu_j}-n\int_{-1}^1\frac{\kappa u(x)}
{b_0\kappa-x}\,dx\right|
\notag\\
&=\left|\sum_{j=1}^n\psi_{b_0}(\nu_j/n)-n\int_{-1}^1\psi_{b_0}(x)u(x)\,dx\right|\le n^{1/6}
\notag
\end{align}
since $\nu\in H_n$. Hence,
\begin{equation}
\left|\frac{a(\nu)S_n-\gamma_n}{n^{1/3}}\right|\le \frac 1{n^{1/3}}|a-a_0|S_n+
\frac 1{n^{1/3}}|a_0S_n-\gamma_n|\le Cn^{-1/6}.
\notag
\end{equation}
This proves (\ref{3.27}).

If $n$ is so large that $\kappa|b-b_0|\le\epsilon$ for all $\nu\in H_n$, which we can achieve by (\ref{3.25}),
then using $|\nu_j/n|\le1+\epsilon$ we get
\begin{align}
\kappa(b-\nu_j/S_n)&=\kappa(b-b_0)+\kappa b_0-(1+\epsilon)+1+\epsilon-\nu_j/n
\notag\\
&\ge\kappa b_0-(1+\epsilon)+1+\epsilon-\nu_j/n-\kappa|b-b_0|\ge\epsilon,
\notag
\end{align}
so we have $b-\nu_j/S_n\ge\epsilon/\kappa\doteq\alpha_0$. Furthermore,
\begin{equation}
\kappa(b-\nu_j/S_n)=\kappa(b-b_0)+\kappa b_0-\nu_j/n\le\epsilon+\kappa b_0+1+\epsilon\doteq\kappa \beta_0.
\notag
\end{equation}
Thus, $H_n\subseteq F_n$ with these choices of $\alpha_0$ and $\beta_0$.

Finally, we want to control $|d(\nu)-\delta|$. By (\ref{3.3}) and (\ref{3.23}),
\begin{align}\label{3.29}
&d^3-\delta^3=\sum_{j=1}^n\frac{S_n^2}{(bS_n-\nu_j)^3}-\int_{-1}^1\frac{\kappa^3u(x)}{(b_0\kappa-x)^3}\,dx
\notag\\
&=\sum_{j=1}^n\left(\frac{S_n^2}{(bS_n-\nu_j)^3}-\frac{S_n^2}{(b_0S_n-\nu_j)^3}\right)
+\frac{\kappa^2}{2n}\left(\sum_{j=1}^n\psi_{b_0}''(\nu_j/n)-
\int_{-1}^1\psi_{b_0}''(x)u(x)\,dx\right)
\end{align}
Since $\nu\in H_n$,
\begin{equation}\label{3.30}
\frac{\kappa^2}{2n}\left|\sum_{j=1}^n\psi_{b_0}''(\nu_j/n)-
\int_{-1}^1\psi_{b_0}''(x)u(x)\,dx\right|\le\frac{\kappa^2}2 n^{-5/6}.
\end{equation}
Using $b-\nu_j/n\in[\alpha_0,\beta_0]$ and (\ref{3.25}) we see that
\begin{equation}\label{3.31}
\left|\sum_{j=1}^n\left(\frac{S_n^2}{(bS_n-\nu_j)^3}-\frac{S_n^2}{(b_0S_n-\nu_j)^3}\right)\right|
\le Cn^{-5/6}.
\end{equation}
Since $|d^3-\delta^3|\ge |d-\delta|\delta^2$, (\ref{3.29}), (\ref{3.30}) and (\ref{3.31}) give
$|d(\nu)-\delta|\le Cn^{-5/6}$. We see that $H_n\subseteq G_n$, which is what we wanted to prove.

It remains to prove lemma \ref{lem3.3}. For this we will use the following estimate.

\begin{lemma}\label{lem3.4}
Let $\{\nu_j\}$ be the eigenvalues of $\sqrt{n} X$, where $X$ is an Hermtian Wigner matrix with finite
fourth moments. Assume that $\phi\in C_0^\infty(\mathbb{R})$ is real-valued and let $\epsilon_0\in (0,1)$
be given. Then there is a constant $C$, depending on $\phi $ and $\epsilon_0$, so that
\begin{equation}\label{3.34}
\mathbb{E}_X\left[\left(\sum_{j=1}^n\phi(\nu_j/n)-n\int_{-1}^1\phi(x)u(x)\,dx\right)^2\right]\le
Cn^{\epsilon_0}.
\end{equation}
\end{lemma}

Before we prove lemma \ref{lem3.4} we will use it to prove lemma \ref{lem3.3}.

\begin{proof} ({\it of lemma \ref{lem3.3}}) It follows from theorem 2.12 in \cite{Bai2} that 
$\mathbb{P}[H_n']\to 1$ as $n\to\infty$. If $\phi=\psi_{2b_0}', \psi_{b_0}, \psi_{b_0}'$ or 
$\psi_{b_0}''$ then $\phi\in C_0^\infty(\mathbb{R})$ and lemma \ref{lem3.4} with $\epsilon_0=1/6$ gives
\begin{equation}
\mathbb{P}_{\nu}\left[\left|\sum_{j=1}^n\phi(\nu_j/n)-n\int_{-1}^1\phi(x)u(x)\,dx\right|\ge n^{1/6}\right]
\le \frac{Cn^{1/6}}{n^{1/3}},
\notag
\end{equation}
by Chebyshev's inequality. This proves lemma \ref{lem3.3}.
\end{proof}

\begin{proof} ({\it of lemma \ref{3.4}}). Pick $A>0$. There is a function $\psi_A\in C_0^\infty(\mathbb{R})$
such that $0\le \psi_A\le 1$, $\psi_A(x)=1$ if $|x|\le A$, $\text{supp}\,\psi_A\subseteq[-(A+1),A+1]$ and
$|\psi_A^{(r)}(x)|\le c_m$ for all $x\in\mathbb{R}$, $0\le r\le m$, where the constant $c_m$ is independent of $A$.
For $z\in\mathbb{C}$ we define
\begin{equation}\label{3.35}
\phi_A(z)=\frac 1{2\pi}\int_{-\infty}^\infty\psi_A(\xi)\hat{\phi}(\xi)e^{i\xi x}\,d\xi,
\end{equation}
which is an entire function of $z$. Here,
\begin{equation}\label{3.36}
\hat{\phi}(\xi)=\int_{-\infty}^\infty e^{-i\xi x}\phi(x)\,dx
\end{equation}
is the Fourier transform of $\phi$. The function $\phi_A$ has the following properties. There is a constant
$C$ independent of $A$ so that
\begin{equation}\label{3.37}
|\phi_A(z)|\le\frac{Ce^{(A+1)|\im z|}}{|z|^2}
\end{equation}
if $z\neq 0$, and
\begin{equation}\label{3.38}
|\phi_A(z)|\le Ce^{(A+1)|\im z|}
\end{equation}
for all $z$. Furthermore, given $m\ge 1$, there is a constant $C_m$ so that
\begin{equation}\label{3.39}
|\phi(x)-\phi_A(x)|\le \frac{C_m}{A^m}
\end{equation}
for all $x\in\mathbb{R}$. The inequality (\ref{3.38}) follows immediately from (\ref{3.35}) and 
$|\hat{\phi}|\le ||\phi||_\infty$. Integration by parts gives
\begin{equation}
\phi_A(z)=\frac 1{2\pi(iz)^2}\int_{-\infty}^\infty \frac{d^2}{d\xi^2}(\psi_A(\xi)\hat{\phi}(\xi))e^{i\xi x}\,dx.
\notag
\end{equation}
The properties of $\psi_A$ and suitable estimates of $\hat{\phi}$ and its derivatives obtained from (\ref{3.36})
using integration by parts, now gives (\ref{3.37}). The estimate (\ref{3.39}) is also easy to prove using
integration by parts.

Let $\gamma_\pm$ be given by $t\to\mp t\pm iv$, $t\in\mathbb{R}$, where $v>0$ is fixed, and let $\gamma=\gamma_+
+\gamma_-$. Cauchy's integral formula and the estimate (\ref{3.37}) show that we can represent $\phi_A$ by
\begin{equation}\label{3.40}
\phi_A(x)=\frac 1{2\pi i}\int_{\gamma}\frac{\phi_A(z)}{z-x}\,dz.
\end{equation}
We now turn to the proof of (\ref{3.34}). Write $r_A=\phi-\phi_A$. Then,
\begin{align}\label{3.41}
&\mathbb{E}_X\left[\left(\sum_{j=1}^n\phi(\nu_j/n)-n\int_{-1}^1\phi(x)u(x)\,dx\right)^2\right]^{1/2}
\notag\\
&\le \mathbb{E}_X\left[\left(\sum_{j=1}^n\phi_A(\nu_j/n)-n\int_{-1}^1\phi_A(x)u(x)\,dx\right)^2\right]^{1/2}
\notag\\
&+\mathbb{E}_X\left[\left(\sum_{j=1}^n r_A(\nu_j/n)-n\int_{-1}^1r_A(x)u(x)\,dx\right)^2\right]^{1/2}.
\end{align}
The second term in the right hand side of (\ref{3.41}) is $\le nC_m/A^m$ by (\ref{3.39}), which is 
$\le n^{\epsilon_0/2}$ if $A=(C_mn^{1-\epsilon_0/2})^{1/m}$. In order to estimate the first term we need
the following two lemmas.

\begin{lemma}\label{lem3.5}
Assume that $X$ is an Hermitian Wigner matrix with finite fourth moments and $m_n$ is given by (\ref{2.20}).
Then there is a constant $C$ so that
\begin{equation}\label{3.42}
\mathbb{E}_X[|m_n(z)-\mathbb{E}_X[m_n(z)]|^2]\le \frac C{n^2|\im z|^4}
\end{equation}
for all $z$ with $\im z\neq 0$.
\end{lemma}

\begin{lemma}\label{lem3.6}
Let $r_n$ be a sequence of positive numbers tending to zero. 
Then there is a constant $C$ such that if $n$ is sufficiently large,
\begin{equation}\label{3.43}
|\mathbb{E}_X[m_n(z)]-m(z)|\le \frac C{n|\im z|^5}
\end{equation}
for all $z$ such that $(nr_n)^{-1/5}\le |\im z|\le 1$. Here $m(z)$ is given by (\ref{2.21}).
\end{lemma}

These two lemmas can be extracted from \cite{BMT} (lemma 2.5) and \cite{Bai1}, but for completeness and 
convenience we give somewhat streamlined proofs in section \ref{sect3.4}.

Combining (\ref{3.42}) and (\ref{3.43}) we get
\begin{equation}\label{3.44}
\mathbb{E}_X[|m_n(z)-m(z)|^2]^{1/2}\le \frac C{n|\im z|^5}
\end{equation}
if $(nr_n)^{-1/5}\le |\im z|\le 1$. Now, by (\ref{3.40}),
\begin{align}
&\mathbb{E}_X\left[\left(\sum_{j=1}^n\phi_A(\nu_j/n)-n\int_{-1}^1\phi_A(x)u(x)\,dx\right)^2\right]
\notag\\
&\le\frac{n^2}{4\pi^2}\int_\gamma\,|dz|\int_\gamma\,|dw||\phi_A(z)||\phi_A(w)|
\mathbb{E}_X[|m_n(z)-m(z)|^2]^{1/2}\mathbb{E}_X[|m_n(w)-m(w)|^2]^{1/2}
\notag\\
&\le\frac C{v^{10}}e^{2A}
\notag
\end{align}
by (\ref{3.37}), (\ref{3.38}) and (\ref{3.44}) provided 
$(nr_n)^{-1/5}\le v\le 1$. Hence, the first term in the right hand side of (\ref{3.41}) is
$\le Cv^{-5}e^{Av}$. We need $Av\le 1$, which gives $v\le 1/A=(C_mn^{1-\epsilon_0/2})^{-1/m}$. Also,
we need $v^{-5}\le n^{\epsilon_0/2}$, i.e. $v\ge n^{-5\epsilon_0/2}$. Take $v=n^{-\delta_0}$, where
$\delta_0=\min(1/10,\epsilon_0)$, $r_n=n^{-1/2}$ and $m$ so large that $m\delta_0\ge 1$. Then all the
required inequalities are satified and we have proved \ref{lem3.4}.
\end{proof}

\subsection{The correlation kernel at the edge}\label{sect3.3}

In this section we will prove theorem \ref{thm3.1}.
Let
\begin{equation}\label{3.100}
f(z)=\frac {z^2}2-az+\frac 1S\sum_{j=1}^n\log (Sz-\nu_j).
\end{equation}
Then, by (\ref{2.33:2}),
\begin{equation}\label{3.101}
K_{n,S}^\nu (aS+u,aS+v)=\frac 1{(2\pi i)^2}\int_{\gamma_L}\,dz\int_{\Gamma_M}\,dw
\frac{e^{-vw+uz}}{w-z}e^{S(f(w)-f(z))}.
\end{equation}
Note that $a$ and $b$ are chosen so that $f'(b)=f''(b)=0$. We can now argue as in section \ref{sect2.3}
in order to find good contours. Define $g(t)=\re f(x(t)+iy(t))$, where $x(0)=b$, $y(0)=0$. Then
\begin{align}\label{3.102}
&g'=xx'-yy'-ax'+\sum_{j=1}^n\frac{S(xx'+yy')-\nu_jx'}
{(Sx-\nu_j)^2+S^2y^2}
\notag\\
&=\sum_{j=1}^n\left[\frac{S(xx'-yy')-2bSx'+x'\nu_j}{(bS-\nu_j)^2}+
\frac{S(xx''+yy'-\nu_jx'}
{(Sx-\nu_j)^2+S^2y^2}\right],
\end{align}
where we have used (\ref{3.1}) and (\ref{3.2}) in the second equality.
If we write the expression in the last sum in (\ref{3.102}) as on fraction the numerator becomes
\begin{align}
&S^2(-x^2x'+2yy'x+y^2x'+2bxx'-2byy'-b^2x')\nu_j
\notag\\
&+S^3((xx'-yy'-2bx')(x^2+y^2)+b^2(xx'+yy')).
\notag
\end{align}
We want to choose $x(t)+iy(t)$ so that the expression in the numerator is independent of $\nu_j$ which gives
the equation
\begin{equation}
-\frac 13x^3+y^2x+bx^2-by^2-bx=C.
\notag
\end{equation}
Since $x(0)=b$, $y(0)=0$ we see that $C=-b^3/3$, and we obtain
\begin{equation}
y^2(x-b)=\frac 13(x-b)^3.
\notag
\end{equation}
We see that $x(t)=b$ is one possibility and $y(t)=\pm\frac 1{\sqrt{3}}(x(t)-b)$ another. The choice
$x(t)=b$, $y(t)=t$ gives
\begin{equation}\label{3.103}
g'(t)=-\sum_{j=1}^n\frac{S^3t^3}{(bS-\nu_j)^2((bS-\nu_j)^2+S^2t^2)}
\end{equation}
and the choice $x(t)=t$, $y(t)=\pm\frac 1{\sqrt{3}}(t-b)$ gives
\begin{equation}\label{3.104}
g'(t)=-\sum_{j=1}^n\frac{S^3t^3}{(bS-\nu_j)^2((bS-\nu_j+St)^2+S^2t^2/3)}.
\end{equation}

This leads us to the following choice of contours. Let $\gamma$ be given by $z(t)$, where
\begin{equation}\label{3.105}
z(t)=\begin{cases} b+te^{\pi i/6}, & t\le 0 \\
b+te^{5\pi i/6}, & t\ge 0,
\end{cases}
\end{equation}
and let $\Gamma$ be given by $w(s)=b+is$, $s\in\mathbb{R}$. We can deform the contour $\gamma_L$ in
(\ref{3.101}) to $\gamma$ and $\Gamma_M$ to $\Gamma$.

From (\ref{3.1}), (\ref{3.103}) and $\nu\in F_n$ we see that for $t\ge 0$
\begin{equation}
g'(t)\le -\sum_{j=1}^n\frac{St^3}{(bS-\nu_j)^2(\beta_0^2+t^2)}=-\frac{t^3}{\beta_0^2+t^2}
\notag
\end{equation}
and similarly for $t\le 0$,
\begin{equation}
g'(t)\ge -\frac{t^3}{\beta_0^2+t^2}.
\notag
\end{equation}
From this it follows that
\begin{equation}\label{3.106}
\re f(w(s))-f(b)\le\begin{cases} -s^4/8\beta_0^2 & \text{for $0\le |s|\le\beta_0$}\\
(\beta_0^2-2s^2)/8 & \text{for $|s|\ge\beta_0$}.
\end{cases}
\end{equation}
Using the fact that $(Sb-\nu_j+St)^2\le 2\beta_0^2 S^2+2S^2t^2$ we get in a similar way from (\ref{3.104}), that
\begin{equation}\label{3.107}
f(b)-\re f(z(t))\le
\begin{cases} -t^4/24\beta_0^2 &\text{for $0\le |s|\le\beta_0$}\\
(\beta_0^2-2t^2)/24            &\text{for $|s|\ge\beta_0$.}
\end{cases}
\end{equation}
Set $\epsilon=S^{-5/24}$ and let $I_1=(-\infty,\epsilon]$, $I_2=[-\epsilon,\epsilon]$, 
$I_3=[\epsilon,\infty)$. Define $\Gamma_k$ by $w(s)$, $s\in I_k$ and $\gamma_k$ by $z(t)$, $t\in I_k$. Let
\begin{equation}
I_{jk}=\frac{e^{(v-u)b}dS^{1/3}}{(2\pi i)^2}\int_{\gamma_j}\,dz\int_{\Gamma_k}\,dw
\frac{e^{-vw+uz}}{w-z}e^{S(f(w)-f(z))},
\notag
\end{equation}
where $u=dS^{1/3}\xi$, $v=dS^{1/3}\eta$. Then
\begin{equation}
dS^{1/3}e^{(\eta-\xi)dS^{1/3}}K_{n,S}^\nu(aS+dS^{1/3}\xi,aS+dS^{1/3}\eta)=\sum_{j,k=1}^3 I_{jk}.
\notag
\end{equation}
We first show that
\begin{equation}\label{3.108}
|I_{1,k}|,|I_{3,k}|\le Ce^{-cS^{1/8}\xi-cS^{1/6}}
\end{equation}
for $S\ge 1$, $k=1,2,3$. Consider $I_{3,k}$, the estimation of $I_{1,k}$ is analogous. If $z\in\gamma_3$ and 
$w\in\Gamma$, then
\begin{equation}\label{3.109}
\left|\frac{e^{-\eta dS^{1/3}(w-b)+\xi dS^{1/3}(z-b)}}{w-z}\right|\le
\frac C{\epsilon}e^{-C\epsilon dS^{1/3}\xi}\le \frac C{\epsilon}e^{-CS^{1/8}\xi}.
\end{equation}
Here we have used the fact that
\begin{equation}
\sum_{j=1}^n\frac{S^2}{(bS-\nu_j)^3}=\sum_{j=1}^n\frac{S}{(bS-\nu_j)^2}\frac 1{b-\nu_j/S}\in
[1/\beta_0,1/\alpha_0],
\notag
\end{equation}
by(\ref{3.1}) and (\ref{3.4}), which gives $d\in [1/\beta_0^{1/3},1/\alpha_0^{1/3}]$. It follows from
(\ref{3.106}) that
\begin{equation}\label{3.110}
\int_{-\infty}^\infty e^{S(\re f(w(s))-f(b))}\,ds\le \frac C{S^{1/4}}.
\end{equation}
Furthermore, (\ref{3.107}) gives
\begin{equation}\label{3.111}
\int_{\epsilon}^\infty e^{S(f(b)-\re f(z(t)))}\,dt\le \frac C{S^{1/4}}e^{-CS\epsilon^4}
\le\frac C{S^{1/4}}e^{-CS^{1/6}}.
\end{equation}
If we combine (\ref{3.108}), (\ref{3.110}) and (\ref{3.111}) we get (\ref{3.108}).

Next, we show that there are positive constants $C, c, S_0$ such that for $S\ge S_0$,
\begin{equation}\label{3.112}
|I_{2,k}|\le Ce^{-\xi}e^{-cS^{1/6}}.
\end{equation}
We treat $I_{2,3}$, the proof for $I_{2,1}$ is analogous. We have that
\begin{align}\label{3.113}
I_{2,3}=\frac{dS^{1/3}}{(2\pi i)^2}&\int_{-\epsilon}^\epsilon\,dt\int_{-\epsilon}^{\infty}\,ds 
\frac{z'(t)}{w(s)-z(t)}
\notag\\
&\times e^{-\eta dS^{1/3}(w(s)-b)+\xi dS^{1/3}(z(t)-b)+
S(f(w(s))-f(b))+S(f(b)-f(z(t)))}.
\end{align}
\begin{claim}\label{clm3.7}
If $|z-b|\le\alpha_0/2$, then
\begin{equation}\label{3.114}
f(z)=f(b)+\frac 13d^3(z-b)^3-\lambda d^4(z-b)^4+R(z-b),
\end{equation}
where
\begin{equation}\label{3.115}
|R(z)|\le 20\alpha_0^{-5}|z|^5
\end{equation}
and $\lambda\in[(\alpha_0^{2/3}/\beta_0)^2/4, (\beta_0^{2/3}/\alpha_0)^2/4]$.
\end{claim}

\begin{proof}
Let $h(t)=f(b+t(z-b)$. Then Taylor's formula yields (\ref{3.114}) with $\lambda=-f^{(4)}(b)/24d^4$ and
\begin{equation}
R(z)=\frac {z^5}{120}f^{(5)}(b)+\frac {z^5}{120}\int_0^1(1-t)^5f^{(5)}(b+t(z-b))\,dt.
\notag
\end{equation}
Now, by (\ref{3.1}) and (\ref{3.4}),
\begin{equation}
-\frac1{24}f^{(4)}(b)=\frac 14\sum_{j=1}^n \frac{S^3}{(bS-\nu_j)^4}\in [\frac 1{4\beta_0^2},\frac 1{4\alpha_0^2}], 
\notag
\end{equation}
and similarly $d^4\in
[1/\beta_0^{4/3},1/\alpha_0^{4/3}]$. Hence, the result for $\lambda$ follows. If $|z-b|\le\alpha_0/2$, then
$|S(b+t(z-b))-\nu_j|\ge |bS-\nu_j|/2$ and thus, by (\ref{3.1}) and (\ref{3.4}),
\begin{equation}
|f^{(5)}(b+t(z-b))|\le 24\cdot 2^5\sum_{j=1}^n\frac{S^4}{(bS-\nu_j)^5}\le \frac{24\cdot 2^5}{\alpha_0^3}.
\notag
\end{equation}
This gives (\ref{3.115}).
\end{proof}

Using (\ref{3.105}), (\ref{3.114}) and making the change of variables $\tau=dS^{1/3}t$, the $t$-integral 
in (\ref{3.113}) becomes
\begin{align}\label{3.116}
&e^{5\pi i/6}\int_0^{\epsilon'}\frac 1{idS^{1/3}s-\tau e^{5\pi i/6}} e^{i\xi\tau e^{\pi i/3}-i\tau^3/3+
\lambda\tau^4 e^{4\pi i/3}/S^{1/3}-R_S(\tau e^{5\pi i/6})}\,d\tau
\notag\\
&+e^{\pi i/6}\int_{-\epsilon'}^0\frac 1{idS^{1/3}s-\tau e^{\pi i/6}} e^{i\xi\tau e^{-\pi i/3}-i\tau^3/3+
\lambda\tau^4 e^{2\pi i/3}/S^{1/3}-R_S(\tau e^{\pi i/6})}\,d\tau,
\end{align}
where $\epsilon'=dS^{1/3}\epsilon=d S^{1/8}$ and $R_S(\tau)=SR(\tau/dS^{1/3})$.
Let $C'_+$ be the curve from $0$ to $\epsilon'$ consisting of the line segments from $0$ to $-i$, from $-i$ to
$\epsilon'-i$ and from $\epsilon'-i$ to $\epsilon'$, and
$C'_-$ the curve from $-\epsilon'$ to $0$ consisting of the line segments from $-\epsilon'$ 
to $-\epsilon'-i$ , from $-\epsilon'-i$ to
$-i$ and from $-i$ to $0$. Now, let $C_-$ be the curve obtained from $C'_-$ by rotating it around the origin by 
an angle $-\pi/3$, and let $C_+$ be the curve obtained from $C'_+$ by rotating it around the origin
an angle $\pi/3$. The sum of the two integrals in (\ref{3.116}) can then be written
\begin{equation}\label{3.117}
i\int_{C_-+C_+}\frac 1{idS^{1/3}-iz}e^{i\xi z+iz^3/3+\lambda z^4/S^{1/3}-R_S(iz)}\,dz.
\end{equation}
The contour $C_-+C_+$ can be deformed into $C_1+C_2+C_3+C_4+C_5$, where

$C_1$ is the line segment from $\epsilon' e^{2\pi i/3}$ to $\epsilon' e^{2\pi i/3}+e^{-5\pi i/6}$,

$C_2$ is the line segment from $\epsilon' e^{2\pi i/3}+e^{-5\pi i/6}$ to $-\sqrt{3}+i$,

$C_3$ is the line segment from $-\sqrt{3}+i$ to $\sqrt{3}+i$,

$C_4$ is the line segment from $\sqrt{3}+i$ to $\epsilon' e^{\pi i/3}+e^{-\pi i/6}$ and

$C_5$ is the line segment from $\epsilon' e^{\pi i/3}+e^{-\pi i/6}$ to $\epsilon' e^{\pi i/3}$.

The integral in (\ref{3.117}) can then be written
\begin{equation}\label{3.118}
i\sum_{j=1}^5\int_{C_j}\frac 1{idS^{1/3}-iz}e^{i\xi z+iz^3/3+\lambda z^4/S^{1/3}-R_S(iz)}\,dz.
\end{equation}
Combining this with (\ref{3.113}) now leads us to the estimate
\begin{equation}
|I_{2,3}|\le\frac{dS^{1/3}}{4\pi^2}\sum_{j=1}^5\int_{\epsilon}^\infty\,ds\int_{C_j}\,|dz|
\frac{e^{\re S(f(b+is)-f(b))}}{|dS^{1/3}s-z|}e^{\re(i\xi z+iz^3/3+\lambda z^4/S^{1/3}-R_S(iz))}.
\notag
\end{equation}
Note that $|dS^{1/3}s-z|\ge\epsilon'/2$ when $s\ge\epsilon$ and $z\in C_j$. Also, by (\ref{3.106}),
\begin{equation}\label{3.119}
\int_{\epsilon}^\infty e^{\re S(f(b+is)-f(b))}\,ds\le\frac{C}{S^{1/4}}e^{-cS\epsilon^4}=\frac{C}{S^{1/4}}
e^{-cS^{1/6}}.
\end{equation}
The contour $C_3$ is given by $z(t)=t+i$, $|t|\le\sqrt{3}$. This gives, using (\ref{3.115}),
\begin{equation}\label{3.12-}
\int_{C_3}e^{\re(i\xi z+iz^3/3+\lambda z^4/S^{1/3}-R_S(iz))}\,|dz|\le Ce^{-\xi}
\end{equation}
for $S\ge 1$. The curve $C_4$ is given by $z(t)=e^{-\pi i/6}+te^{\pi i/3}$, $\sqrt{3}\le t\le 
\epsilon'$. Then $\re z(t)^4\le 0$ if $t\ge\sqrt{3}$, $\re (i\xi z(t)+iz(t)^3/3)\le-\xi+1/3-t^2$ and
$|R_S(iz(t))|\le CS^{-1/24}$. This gives
\begin{equation}\label{3.121}
\int_{C_4}e^{\re(i\xi z+iz^3/3+\lambda z^4/S^{1/3}-R_S(iz))}\,|dz|\le Ce^{-\xi}.
\end{equation}
The curve $-C_5$ is given by $z(t)=\epsilon' e^{\pi i/3}+t e^{-\pi i/6}$, $0\le t\le 1$, and inserting the 
parametrization and estimating we see that we get an estimate
\begin{equation}
\int_{C_5}e^{\re(i\xi z+iz^3/3+\lambda z^4/S^{1/3}-R_S(iz))}\,|dz|\le Ce^{-\xi}
\notag
\end{equation}
if $\epsilon'\ge c_0$, where $c_0$ is a numerical constant. This holds if
$S\ge S_0=(c_0\alpha_0^{1/3})^8\ge (C_0/d)^8$. The estimates for the integrals along $C_1$ and $C_2$ are 
analogous to the estimates for $C_5$ and $C_4$ respectively. Collecting all the estimates we have proved
(\ref{3.112}).

It remains to estimate and compute the asymptotics of $I_{22}$. Let $C'$ be the contour $t\to t+i$,
$|t|\le\epsilon'$ and let $C=C_1+C_2+C_3+C_4+C_5$. The same type of computations that led to the
expression (\ref{3.117}) now gives
\begin{equation}\label{3.122}
I_{22}=-\frac{i}{(2\pi i)^2}\int_C\,dz\int_{C'}\,dw\frac{e^{i\eta w+i\xi z}}{z+w}e^{iw^3/3-\lambda w^4/S^{1/3}
+R_S(-iw)}e^{iz^3/3+\lambda z^4/S^{1/3}-R_S(iz)}.
\end{equation}
By introducing the parametrizations of $C$ and $C'$ we can now again prove that
\begin{equation}\label{3.123}
|I_{22}|\le Ce^{-(\xi+\eta)},
\end{equation}
for $S\ge S_0$ with a suitable $S_0$ that only depends on $\alpha_0$. Combining the estimates (\ref{3.108}),
(\ref{3.112}) and (\ref{3.123}) we obtain (\ref{3.5}).

We now take $S=S_n=\kappa n$. It is clear from (\ref{3.108}) and (\ref{3.112}) that all contributions 
except $I_{22}$ go to zero uniformly for $\xi,\eta$ in a compact set and all $\nu\in F_n$ as $n\to\infty$.
Let $\tilde{C}$ be the ``limit'' of $C$ as $n\to\infty$, i.e. $\tilde{C}=\tilde{C_1}+\tilde{C_2}+\tilde{C_3}$,
where $-\tilde{C_1}:-\sqrt{3}+i+te^{2\pi i/3}$, $t\ge 0$, $\tilde{C_2}:t+i$, $|t|\le\sqrt{3}$
$\tilde{C_3}:\sqrt{3}+i+te^{\pi i/3}$, $t\ge 0$. Introducing the parametrizations into the integral in
(\ref{3.122}) we see that we can let $n\to\infty$ in (\ref{3.122}) with $S=S_n$, to obtain
\begin{equation}\label{3.124}
\lim_{n\to\infty} I_{22}=
-\frac{i}{(2\pi i)^2}\int_{\tilde{C}}\,dz\int_{\im w=1}\,dw\frac{e^{i\eta w+i\xi z}}{z+w}e^{i(w^3+z^3)/3}
\end{equation}
uniformly for $\xi,\eta$ in a compact set and all $\nu\in F_n$. A deformation argument now shows that we can deform
$\tilde{C}$ to $\im z=1$, and in this way we see that the right hand side of (\ref{3.124}) equals
the Airy kernel (\ref{1.7}), see e.g. proposition 2,3 in \cite{JDPNG}. This proves (\ref{3.6}) and finishes the
proof of theorem \ref{thm3.1}.

\subsection{Proofs of some lemmas}\label{sect3.4}
The proofs of lemma 3.5 and lemma 3.6 can be extracted from \cite{BMT} and \cite{Bai1}. The presentation below
is somewhat streamlined for our purposes. We use notation similar to that in \cite{BMT} and \cite{Bai1}.

Recall that $X=(x_{ij})$ is an Hermitian Wigner matrix, such that $\mathbb{E}[|x_{ij}|^2]=\sigma^2$ and
$\mathbb{E}[|x_{ij}|^4]\le K$ for all $1\le i\le j\le n$, where $K<\infty$ is a constant. 
Let $X_k$ be the matrix obtained from $X$ by removing row $k$ and column $k$, and let $\alpha_k$ be column
$k$ of $X$ with element number $k$ removed. Set
\begin{equation}
D=\left(\frac 1{\sqrt{n}}X-zI\right)^{-1},\quad D_k=\left(\frac 1{\sqrt{n}}X_k-zI\right)^{-1}.
\notag
\end{equation}
Write $v=\im z$. We can assume that $v>0$. We need some identities from matrix theory.

\begin{lemma}\label{lem3.8} The following identities hold,
\begin{equation}\label{1}
\Tr D=\sum_{k=1}^n\frac 1{x_{kk}/\sqrt{n}-z-\alpha_k^*D_k\alpha_k},
\end{equation}
\begin{equation}\label{2}
\Tr D-\Tr D_k=\sum_{k=1}^n\frac {1+\frac 1n\alpha_k^*D_k^2\alpha_k}{x_{kk}/\sqrt{n}-z-\alpha_k^*D_k\alpha_k}.
\end{equation}
\end{lemma}

\begin{proof}
The identity (\ref{1}) follows from Cramer's rule and the formula
\begin{equation}
\det\left(\begin{matrix} A & B\\C & D\end{matrix}\right)=\det(A)\det(D-CA^{-1}B),
\notag
\end{equation}
which holds whenever $A$ is invertible. The formula (\ref{2}) follows from the formula
\begin{equation}
\left(\begin{matrix} A & B\\C & D\end{matrix}\right)^{-1}
=\left(\begin{matrix} A^{-1}+A^{-1}B(D-CA^{-1}B)^{-1}CA^{-1} &
-A^{-1}B(D-CA^{-1}B)^{-1} \\ -(D-CA^{-1}B)^{-1}CA^{-1} & (D-CA^{-1}B)^{-1}\end{matrix}\right)
\notag
\end{equation}
for the inverse of a block matrix.
\end{proof}

Let
\begin{align}
\beta_k&=-x_{kk}/\sqrt{n}+z+\alpha_k^*D_k\alpha_k,
\notag\\
\beta_k^*&=z+\frac{\sigma^2}n\Tr D_k,
\notag\\
\beta&=z+\frac{\sigma^2}n\Tr D,
\notag\\
\epsilon_k^*&=\beta_k-\beta_k^*=-x_{kk}/\sqrt{n}+\frac 1n(\alpha_k^*D_k\alpha_k-\sigma^2\Tr D_k).
\notag
\end{align}
Let $\mathbb{E}_k$ denote expectation with respect to the elemts in row/colum $k$ in $X$.
We need the following basic estimates.

\begin{lemma}\label{lem3.9}
\begin{equation}\label{3}
\im\beta_k=v(1+\frac 1n\alpha_k^*D_kD_k^*\alpha_k)\ge v,
\end{equation}
\begin{equation}\label{4}
\im\beta_k^*=v(1+\frac 1n\Tr D_kD_k^*)\ge v,
\end{equation}
\begin{equation}\label{5}
|1+\frac 1n\alpha_k^*D_k^2\alpha_k|\le 1+\frac 1n\alpha_k^*D_kD_k^*\alpha_k,
\end{equation}
\begin{equation}\label{6}
|\Tr D-\Tr D_k|\le \frac 1v,
\end{equation}
\begin{equation}\label{7}
\mathbb{E}_k[|\alpha_k^*D_k\alpha_k-\sigma^2\Tr D_k|^2]\le K\Tr D_kD_k^*,
\end{equation}
\begin{equation}\label{8}
\mathbb{E}_k[|\alpha_k^*D_k^2\alpha_k-\sigma^2\Tr D_k^2|^2]\le K\Tr D_k^2{D_k^*}^2.
\end{equation}
\end{lemma}

\begin{proof}
We see that
\begin{equation}
\im\beta_k=v+\frac 1{2in}(\alpha_k^*D_k\alpha_k-\alpha_k^*D_k^*\alpha_k)
=v(1+\frac 1n\alpha_k^*D_kD_k^*\alpha_k)\ge v,
\notag
\end{equation}
which gives (\ref{3}) and a similar argument proves (\ref{4}). To prove (\ref{5}) we write
$D_k=U^*\text{diag\,}((\lambda_j/\sqrt{n}-z)^{-1})U$, where $\lambda_1,\dots,\lambda_{n-1}$ are the
eigenvalues of $X_k$ and $U$ is unitary. Then,
\begin{align}
&|\alpha_k^*D_k^2\alpha_k|\le\sum_{j=1}^n|\lambda_j/\sqrt{n}-z|^{-2}|(U\alpha_k)_j|^2
\notag\\&=
(U\alpha_k)^*\text{diag\,}((\lambda_j/\sqrt{n}-z)^{-1})\text{diag\,}((\lambda_j/\sqrt{n}-\bar{z})^{-1})U\alpha_k
=\alpha_k^*D_kD_k^*\alpha_k.
\notag
\end{align}
We see from (\ref{2}), (\ref{3}) and (\ref{5}) that
\begin{equation}
|\Tr D-\Tr D_k|=
\frac{|1+\frac 1n\alpha_k^*D_k^2\alpha_k|}{|\beta_k|}\le
\frac{1+\frac 1n\alpha_k^*D_kD_k^*\alpha_k}{v(1+\frac 1n\alpha_k^*D_kD_k^*\alpha_k)}=\frac 1v,
\notag
\end{equation}
which proves (\ref{6}). Let $A=(a_{ij})$ be an $(n-1)\times(n-1)$ matrix that does not depend on
the elements in row/column $k$. Note that
\begin{equation}\label{0}
\mathbb{E}_k[\alpha_k^*A\alpha_k]=\sum_{j=1}^{n-1}\sigma^2 a_{jj}=\sigma^2\Tr A.
\end{equation}
Hence,
\begin{equation}
\mathbb{E}_k[|\alpha_k^*D_k\alpha_k-\sigma^2\Tr D_k|^2=\mathbb{E}_k[\alpha_k^*A^*\alpha_k
\alpha_k^*A\alpha_k]-\sigma^4(\Tr A^*)(\Tr A).
\notag
\end{equation}
Now,
\begin{align}
\mathbb{E}_k[\alpha_k^*A^*\alpha_k\alpha_k^*A\alpha_k]&=\mathbb{E}_k\left[\sum_{i,j,r,s}
(\alpha_k^*)_i\bar{a}_{ji}(\alpha_k)_j(\alpha_k^*)_ra_{rs}(\alpha_k)_s\right]
\notag\\
&\le K\sum_{i}|a_{ii}|^2+\sum_{i\neq j}\sigma^4\bar{a}_{ii}a_{jj}
+\sum_{i\neq j}\sigma^4\bar{a}_{ji}a_{ji}
\notag\\
&=
(K-2\sigma^4)\sum_{i}|a_{ii}|^2+\sigma^4(\Tr A^*)(\Tr A)+\sigma^4\Tr A^*A
\notag\\
&\le K\Tr A^*A+\sigma^4(\Tr A^*)(\Tr A).
\notag
\end{align}
This proves (\ref{7}) and (\ref{8}).
\end{proof}

Let $\mathcal{F}_{k}$ be the $\sigma$-algebra generated by $\im x_{jk}$, $\re x_{jk}$, $k+1\le i\le j\le n$,
$\mathcal{F}_n=\emptyset$.
Define
\begin{equation}
z_k=\mathbb{E}[\Tr D|\mathcal{F}_{k-1}]-\mathbb{E}[\Tr D|\mathcal{F}_{k}].
\notag
\end{equation}
Then,
\begin{equation}\label{9}
\mathbb{E}[|\Tr D-\mathbb{E}[\Tr D]|^2]=\mathbb{E}\left[\sum_{j,k=1}^n\bar{z}_jz_k\right]=
\sum_{k=1}^n\mathbb{E}[|z_k|^2]
\end{equation}
by orthogonality. Since $\Tr D_k$ is independent of the elements in row/column $k$,
\begin{equation}
\mathbb{E}[\Tr D_k|\mathcal{F}_{k-1}]=\mathbb{E}[\Tr D_k|\mathcal{F}_{k}]
\notag
\end{equation}
and hence
\begin{equation}\label{10}
z_k=\mathbb{E}[\Tr D-\Tr D_k|\mathcal{F}_{k-1}]-\mathbb{E}[\Tr D-\Tr D_k|\mathcal{F}_{k}].
\end{equation}
We can now give the 

\begin{proof} ({\it of lemma \ref{lem2.5}}) Note that $m_n(z)=\frac 1n\Tr D$ and
that (\ref{6}), (\ref{9}) and (\ref{10}) give
\begin{equation}
\mathbb{E}[|\Tr D-\mathbb{E}[\Tr D]|^2]\le\sum_{k=1}^n\frac 2{v^2}=\frac{2n}{v^2}.
\notag
\end{equation}
\end{proof}

We turn next to the

\begin{proof} ({\it of lemma \ref{lem3.5}}) From (\ref{2}) we obtain
\begin{align}
\Tr D-\Tr D_k&=(1+\frac 1n\alpha_k^*D_k^2\alpha_k)(\frac 1{\beta_k^*}-\frac 1{\beta_k})-
(1+\frac 1n\alpha_k^*D_k^2\alpha_k)\frac 1{\beta_k^*}
\notag\\
&=\frac{\epsilon_k^*(1+\frac 1n\alpha_k^*D_k^2\alpha_k)}{\beta_k^*\beta_k}-\frac{1+\frac{\sigma^2}n\Tr D_k^2}
{\beta_k^*}-\frac{\alpha_k^*D_k^2\alpha_k-\sigma^2\Tr D_k^2}{n\beta_k^*}.
\notag
\end{align}
Since neither $\Tr D_k^2$ or $\beta_k^*$ depends on row/column $k$, we see that
\begin{equation}
\mathbb{E}\left[\frac{1+\frac{\sigma^2}n\Tr D_k^2}{\beta_k^*}\left|\right.\mathcal{F}_{k-1}\right]
=\mathbb{E}\left[\frac{1+\frac{\sigma^2}n\Tr D_k^2}{\beta_k^*}\left|\right.\mathcal{F}_{k}\right].
\notag
\end{equation}
Hence, from (\ref{10}), we see that
\begin{align}
z_k&=\mathbb{E}\left[\frac{\epsilon_k^*(1+\frac 1n\alpha_k^*D_k^2\alpha_k)}{\beta_k^*\beta_k}
\left|\right.\mathcal{F}_{k-1}\right]-
\mathbb{E}\left[\frac{\epsilon_k^*(1+\frac 1n\alpha_k^*D_k^2\alpha_k)}{\beta_k^*\beta_k}
\left|\right.\mathcal{F}_{k}\right]
\notag\\
&+\mathbb{E}\left[\frac{\alpha_k^*D_k^2\alpha_k-\sigma^2\Tr D_k^2}{n\beta_k^*}
\left|\right.\mathcal{F}_{k-1}\right]-
\mathbb{E}\left[\frac{\alpha_k^*D_k^2\alpha_k-\sigma^2\Tr D_k^2}{n\beta_k^*}
\left|\right.\mathcal{F}_{k}\right].
\notag
\end{align}
Thus,
\begin{equation}
\mathbb{E}[|z_k|^2]\le 2\mathbb{E}
\left[\frac{|\epsilon_k^*|^2|1+\frac 1n\alpha_k^*D_k^2\alpha_k|^2}{|\beta_k^*|^2|\beta_k|^2}\right]
+2\mathbb{E}
\left[\frac{|\alpha_k^*D_k^2\alpha_k-\sigma^2\Tr D_k^2|^2}{n^2|\beta_k^*|^2}\right].
\notag
\end{equation}
We see from (\ref{3}) and (\ref{5}) that
\begin{equation}\label{11}
\frac{|1+\frac 1n\alpha_k^*D_k^2\alpha_k|^2}{|\beta_k|^2}\le\frac 1{v^2}
\end{equation}
and from (\ref{3}) and (\ref{7}) we obtain
\begin{align}\label{12}
\mathbb{E}_k[|\epsilon_k^*|^2]&=\frac{\sigma^2}n+\frac{1}{n^2}
\mathbb{E}_k[|\alpha_k^*D_k^2\alpha_k-\sigma^2\Tr D_k^2|^2]
\notag\\
&\le\frac{\sigma^2}n+\frac{K}{n^2}\Tr D_kD_k^*\le\frac{K+\sigma^2}{vn}\im\beta_k^*.
\end{align}
Consequently, by (\ref{4}), (\ref{11}) and (\ref{12})
\begin{equation}
\mathbb{E}
\left[\frac{|\epsilon_k^*|^2|1+\frac 1n\alpha_k^*D_k^2\alpha_k|^2}{|\beta_k^*|^2|\beta_k|^2}\right]
\le\frac 1{v^3}\mathbb{E}\left[\frac 1{|\beta_k^*|^2}\mathbb{E}_k[|\epsilon_k^*|^2]\right]
\le\frac{K+\sigma^2}{nv^4}.
\notag
\end{equation}
Note that,
\begin{equation}
\Tr D_k^2{D_k^*}^2=\sum_{j=1}^{n-1}\frac 1{|\lambda_j-z|^4}
\le\frac 1{v^2}\sum_{j=1}^{n-1}\frac 1{|\lambda_j-z|^2}
=\frac 1{v^2}\Tr D_kD_k^*
\notag
\end{equation}
and hence, using also (\ref{4}) and (\ref{8}),
\begin{align}
\mathbb{E}
&\left[\frac{|\alpha_k^*D_k^2\alpha_k-\sigma^2\Tr D_k^2|^2}{n^2|\beta_k^*|^2}\right]
\le \frac 1{nv}\mathbb{E}
\left[\frac 1{n|\beta_k^*|}\mathbb{E}_k[|\alpha_k^*D_k^2\alpha_k-\sigma^2\Tr D_k^2|^2]\right]
\notag\\
&\le\frac{K}{nv}\mathbb{E}
\left[\frac 1{|\beta_k^*|}\frac 1n\Tr D_k^2{D_k^*}^2\right]
\le\frac{K}{nv^3}\mathbb{E}
\left[\frac {1+\frac 1n\Tr D_kD_k^*}{\im\beta_k^*}\right]=\frac {K}{nv^4}.
\notag
\end{align}
We see now from (\ref{9}) that
\begin{equation}\label{13}
\mathbb{E}[|\Tr D-\mathbb{E}[\Tr D]|^2]\le\frac{2K+\sigma^2}{v^4}
\end{equation}
\end{proof}

We still have to give the 

\begin{proof} (\it of lemma \ref{lem3.6}\rm) Set
\begin{equation}\label{14}
\delta=\mathbb{E}[m_n(z)]+\frac 1{z+\sigma^2\mathbb{E}[m_n(z)]}.
\end{equation}
We see that
\begin{align}\label{15}
\delta&=\mathbb{E}\left[m_n(z)+\frac 1{z+\sigma^2m_n(z)}-\left(\frac 1{z+\sigma^2m_n(z)}-
\frac 1{z+\sigma^2\mathbb{E}[m_n(z)]}\right)\right]
\notag\\
&=\mathbb{E}\left[m_n(z)+\frac 1{\beta}\right]+\sigma^2\mathbb{E}\left[
\frac{m_n(z)-\mathbb{E}[m_n(z)]}{\beta\mathbb{E}[\beta]}\right].
\end{align}
By (\ref{14}), $\im\beta\ge v$, we obtain
\begin{equation}\label{16}
\left|\sigma^2\mathbb{E}\left[
\frac{m_n(z)-\mathbb{E}[m_n(z)]}{\beta\mathbb{E}[\beta]}\right]\right|\le
\frac{\sigma^2}{v^2}\frac{\sqrt{2K+\sigma^2}}{nv^2}=\frac{\sigma^2\sqrt{2K+\sigma^2}}
{nv^4}.
\end{equation}
Using (\ref{1}) we find
\begin{align}\label{17}
\mathbb{E}\left[m_n(z)+\frac 1{\beta}\right]&=\mathbb{E}\left[\frac 1n\sum_{k=1}^n
\left(\frac 1{\beta}-\frac 1{\beta_k^*}+\frac 1{\beta_k^*}
\frac 1{\beta_k}\right)\right]
\notag\\
&=\mathbb{E}\left[\frac 1n\sum_{k=1}^n\frac{\epsilon_k^*}{\beta_k\beta_k^*}\right]
-\mathbb{E}\left[\frac {\sigma^2}n\sum_{k=1}^n\frac{\Tr D-\Tr D_k}{\beta\beta_k^*}\right].
\end{align}
It follows from (\ref{3}), (\ref{6}) and  $\im\beta\ge v$, that
\begin{equation}\label{18}
\left|\mathbb{E}\left[\frac {\sigma^2}n\sum_{k=1}^n\frac{\Tr D-\Tr D_k}{\beta\beta_k^*}\right]\right|
\le\frac{\sigma^2}{nv^3}.
\end{equation}
Also,
\begin{equation}
\mathbb{E}\left[\frac 1n\sum_{k=1}^n\frac{\epsilon_k^*}{\beta_k\beta_k^*}\right]=
\frac 1n\sum_{k=1}^n\mathbb{E}\left[\frac{\epsilon_k^*}{{\beta_k^*}^2}\right]-
\frac 1n\sum_{k=1}^n\mathbb{E}\left[\frac{{\epsilon_k^*}^2}{\beta_k{\beta_k^*}^2}\right].
\notag
\end{equation}
We see that
\begin{equation}
\mathbb{E}\left[\frac{\epsilon_k^*}{{\beta_k^*}^2}\right]=
\mathbb{E}\left[\frac 1{{\beta_k^*}^2}\mathbb{E}_k[\epsilon_k^*]\right]=0
\notag
\end{equation}
by (\ref{0}). Furthermore, by (\ref{3}), (\ref{4}) and (\ref{12}),
\begin{equation}
\left|\mathbb{E}\left[\frac{{\epsilon_k^*}^2}{\beta_k{\beta_k^*}^2}\right]\right|\le\frac{K+\sigma^2}{nv^3}.
\notag
\end{equation}
Combining this with (\ref{18}), we see from (\ref{15}), (\ref{16}) and (\ref{17}) that
\begin{equation}\label{19}
|\delta|\le\frac{K+2\sigma^2}{nv^3}+\frac{\sigma^2\sqrt{2K+\sigma^2}}{nv^4}\le\frac{C_0}{nv^4}
\end{equation}
if $v\le 1$.

Solving the equation (\ref{14}) for $\mathbb{E}[m_n(z)]$ we obtain
\begin{equation}\label{20}
\mathbb{E}[m_n(z)]=\frac 1{2\sigma^2}(-z+\sigma^2\delta+\sqrt{(z+\sigma^2\delta)^2-4\sigma^2}).
\end{equation}
Since $\im\mathbb{E}[m_n(z)]\ge 0$ we have to choose the square root with positive imaginary part.
Let
\begin{equation}
m(z)=\frac 1{2\sigma^2}(-z+\sqrt{z^2-4\sigma^2})=
\frac 1{2\sigma^2\pi}\int_{-2\sigma}^{2\sigma}\frac{\sqrt{4\sigma^2-x^2}}{x-z}\,dx.
\notag
\end{equation}
Then,
\begin{equation}\label{21}
|\mathbb{E}[m_n(z)]-m(z)|\le\frac{|\delta|}2+\frac 1{2\sigma^2}|\sqrt{z^2-4\sigma^2}-
\sqrt{z^2-4\sigma^2}|.
\end{equation}
Note that, by (\ref{19}) and the assumption that
$v\ge (nr_n)^{-1/5}$, it follows that $|\delta/v|\le C_0r_n$. For $n$ sufficiently
large, $C_0r_n\le 1/3$ and hence $|\delta/v|\le 1/3$. We now take $\sigma^2=1/4$.
It follows from (\ref{20}) and $\im\mathbb{E}[m_n(z)]\ge 0$ that
$\im\sqrt{(z+\delta/4)^2-1}\ge v$, and similarly we find
$\im\sqrt{z^2-1}\ge v$. Thus,
\begin{align}\label{22}
|\sqrt{(z+\delta/4)^2-1}-\sqrt{z^2-1}|&=\frac{|2\delta z+\delta^2|}
{|\sqrt{(z+\delta/4)^2-1}+\sqrt{z^2-1}|}
\notag\\
&\le\frac{\delta}v(|z|+\delta v)\le\frac{\delta}v (|z|+1).
\end{align}
If $|z|\le 4$, (\ref{19}), (\ref{21}) and (\ref{22}) gives (\ref{3.43}). If 
$|z|\ge 4$, then (\ref{22}) gives
\begin{equation}
|\sqrt{(z+\delta/4)^2-1}-\sqrt{z^2-1}|\le\frac 12|z|,
\notag
\end{equation}
and since $|\sqrt{z^2-1}|\ge\frac 12|z|$, we obtain
\begin{equation}
|\sqrt{(z+\delta/4)^2-1}+\sqrt{z^2-1}|\ge\frac 12|z|.
\notag
\end{equation}
Thus, (\ref{22}) gives
\begin{equation}
|\sqrt{(z+\delta/4)^2-1}-\sqrt{z^2-1}|
\le\frac {2|\delta||z|+\delta^2}{|z|/2}\le 5\frac{\delta}v,
\notag
\end{equation}
and again we get (\ref{3.43}).
\end{proof}
Finally, we prove lemma \ref{lem2.3}.
\begin{proof} Let $\det_2(I-A)$ be the regularized determinant defined for Hilbert-Schmidt operators,
see e.g. \cite{Si}. If $A$ is a trace-class operator, then
\begin{equation}\label{23}
\det(I-A)={\det}_2(I-A)e^{-\Tr A},
\end{equation}
where the left hand side is the Fredholm determinant. Now, see e.g. \cite{Si} ch. 9, for two 
Hilbert-Schmidt operators $A$ and $B$,
\begin{align}\label{24}
|{\det}_2(I-A)-{\det}_2(I-B)|&\le ||A-B||_2e^{\frac 12(||A||_2+||B||_2+1)^2}
\notag\\
&\le ||A-B||_2e^{\frac 12(||A-B||_2+2||B||_2+1)^2}
\end{align}
and
\begin{equation}\label{25}
|{\det}_2(I-A)|\le e^{\frac 12||A||_2+1}.
\end{equation}
Using (\ref{23}) we can write
\begin{align}
{\det}_2(I-A)-{\det}_2(I-B)&=({\det}_2(I-A)-{\det}_2(I-B))e^{-\Tr A}
\notag\\
&+
{\det}_2(I-B)e^{-\Tr B}(e^{-(\Tr A-\Tr B)}-1)
\notag
\end{align}
and the inequality (\ref{2.18}) follows from (\ref{24}) and (\ref{25}).
\end{proof}


\end{document}